\documentclass{amsart}
\usepackage{amsthm,amstext,amsmath,amscd,amssymb,latexsym}
\usepackage{color}
\usepackage[matrix,arrow]{xy}
\usepackage{amsfonts,enumerate,array}
\usepackage[colorlinks=true]{hyperref}
\usepackage{todonotes}
	\usepackage{verbatim}
\usepackage[toc,page]{appendix}
\usepackage{tikz}
\usetikzlibrary{decorations.pathreplacing}

\newcommand{\CC}{{\mathbb C}}
\newcommand{\PP}{\mathbb{P}}

\newcommand{\ZZ}{{\mathbb Z}}

\DeclareMathOperator{\Pic}{Pic}
\DeclareMathOperator{\Cr}{Cr}
\DeclareMathOperator{\Eff}{Eff}

\DeclareMathOperator{\Mov}{Mov}

\DeclareMathOperator{\Cox}{Cox}
\DeclareMathOperator{\adeg}{adeg}

\DeclareMathOperator{\pardeg}{pardeg}

\newcommand{\paper}{: \begin{it}}
\newcommand{\jour }{, \end{it}}

\newtheorem{theorem}{Theorem}[section]
\newtheorem*{theorem*}{Theorem}
\newtheorem{lemma}[theorem]{Lemma}
\newtheorem{proposition}[theorem]{Proposition}
\newtheorem{corollary}[theorem]{Corollary}
\newtheorem{conjecture}[theorem]{Conjecture}

\newtheorem{question}[theorem]{Question}
\theoremstyle{definition}
\newtheorem{definition}[theorem]{Definition}
\newtheorem{notation}[theorem]{Notation}
\newtheorem{remark}[theorem]{Remark}
\newtheorem{problem}[theorem]{Problem}
\newtheorem{example}[theorem]{Example}

\theoremstyle{remark}

\numberwithin{equation}{section}

\definecolor{mygreen}{rgb}{0.13, 0.55, 0.13}

\title{On divisorial $(i)$ classes}

\author{Olivia Dumitrescu}
\address{
Olivia Dumitrescu:
University of North Carolina at Chapel Hill\\
340 Phillips Hall,\\
CB 3250, Chapel Hill,\\
 NC 27599-3250}

\address{and Simion Stoilow Institute of Mathematics\\
Romanian Academy\\
21 Calea Grivitei Street\\
010702 Bucharest, Romania}
\email{dolivia@unc.edu}

\thanks{The first author is supported by NSF grant DMS1802082. The second author is supported by Collaboration Grant 586691 from the Simon's Foundation.}

\author{Nathan Priddis}
\address{Nathan Priddis: Brigham Young University, Provo, Utah 84602, US}
\email{priddis@math.byu.edu}


\begin{document}

\begin{abstract} 
In this paper we introduce and study \emph{divisorial $(i)$ classes} for the blow up of projective space $\mathbb{P}^n$ in several points for $i\in \{-1,0,1\}$. We generalize Noether's inequality, and we  prove that all divisorial $(i)$ classes are in bijective correspondence with the orbit of the Weyl group action on one exceptional divisor following Nagata's original approach.  
Moreover, we prove that the irreducibility condition from the definition of divisorial $(i)$ classes can be replaced by the numerical condition of having positive intersection with all divisorial $(-1)$ classes of smaller degree via the Mukai pairing.
\end{abstract}

\maketitle
{ \hypersetup{linkcolor=black} \tableofcontents}
\setcounter{section}{0}

\newpage
\section*{Introduction}
In this article, we generalize several important results about the blow up of $\PP^n$ in several points. As motivation, let us recall one of Hilbert's problems. 

\subsection{Hilbert's 14th problem}\label{intro} In 1900 Hilbert posed 23 problems at the International Congress of Mathematicians in Paris. Hilbert's fourteenth problem encountered several modifications and generalizations in the subsequent years, one of them due to Zariski in 1953 \cite{zariski}. 
The original formulation of the conjecture (following Nagata \cite{nagataL}) is as follows.

\begin{problem}\label{problem} Let $K$ be a field and $G$ a subgroup of the linear group $GL(s, K)$, acting via automorphisms on the polynomial ring $K[x_1,\ldots, x_s]$. Is the invariant ring $K[x_1,\ldots, x_s]^G$ finitely generated over field $K$?
\end{problem}

The answer of this problem was proven in the affirmative in the following cases: 
\begin{itemize}
\item $K=\CC$ and $G=SL(s, K)$ (Hilbert).
\item $G$ a finite group and $K$ arbitrary (Emmy Noether).
\item $K=\CC$ and $G=\mathbb{G}_{a}$ a one parameter group (Weitzerbock).
\item $K=\CC$ and $G$ a connected semi-simple group (H. Weyl).
\end{itemize}

In 1958 Masayoshi Nagata gave a counterexample to Hilbert's 14th problem (see \cite{nagatab} and \cite{nagata}) by considering the standard unipotent linear action of $\CC^s$ on the polynomial ring in $2s$ variables $S=\CC[x_1,\ldots, x_s, y_1, \ldots, y_s]$ 
\[
(t_1,\ldots, t_s)\cdot(x_1,\ldots, x_s, y_1,\ldots, y_s)=(x_1,\ldots, x_s, y_1+t_1x_1,\ldots, y_s+t_sx_s)
\]

One of the major breakthroughs of Nagata was to show that if $G=\mathbb{G}_{a}^{l}\subset \CC^s$ then the invariant ring $S^G$ is isomorphic to the Cox ring of the blow up $X_{s-l-1,s}$, $S^G\cong \Cox(X_{s-l-1,s})$ (see Section \ref{cox}). In \cite[Theorem 2a]{nagata1} Nagata proves that on $X_{2,s}$ there exists a bijection between $(-1)$ curves (what Nagata calls irreducible exceptional curves of the first type) and the orbit of the Weyl group under the class of an exceptional divisor (what Nagata calls pre-exceptional type of first kind).

In particular, he proved that if $G=\mathbb{G}_{a}^{13}$, the  general linear group in $\CC^{16}$, the invariant ring $S^G$ is not finitely generated. 
This is due to the fact that there are an infinite number of $(-1)$ curves in the Picard group of the blow up of $\PP^2$ in $16$ general points (which we will denote by $X_{2, 16}$), or equivalently that the orbit of the Weyl group $W_{2, 16}$ on one exceptional divisor in $X_{2,16}$ is not finitely generated. We will discuss this more below. 

This theme of research was continued by Mori and led to the development of the Minimal Model
Program and to the identification of Mori Dream Spaces: those for which the Cox ring
is finitely generated, as presented in Section \ref{cox}.
In Section~\ref{s:Nagatacorrespondence} we generalize Nagata's approach to higher dimensions.

Next, we define the main characters in this paper and set some notation.

\begin{notation}
Let $X_{n,s}$ denote the blowup of $\PP^n$ in a collection of $s$ points $p_1,\dots,p_s$ in general position. The Picard group is generated by $\Pic X_{n,s}=\langle H,E_1,\dots,E_s \rangle$, where $H$ is a general hyperplane in $\PP^n$ and $E_i$ is the exceptional divisor corresponding to the point $p_i$. A general divisor $D\in \Pic X_{n,s}$ can be written in terms of degree $d$ and multiplicities $m_i$:
\begin{equation}\label{eq:pic}
D=dH-\sum_{j=1}^s m_j E_j.
\end{equation}
\end{notation}

\subsection{Main results.}

The main goal of this paper is to develop the techniques to prove that the {\it $(i)$ Weyl divisors} are equivalent to {\it divisorial $(i)$ classes} defined by algebraic equations \eqref{eq:num}, below.

Laface and Ugaglia introduced the concept of \emph{$(-1)$ classes} in \cite{LUst}, where they also study properties of the Dolgachev-Mukai form. 
Furthermore, we have used the term \emph{divisorial} $(i)$ classes to emphasize that they are divisors.

\begin{definition} (see also Definition~\ref{def:-1div})\label{intro:-1div}
	For  $i\in \{-1,0,1\}$ we introduce 
	\begin{enumerate} 
\item A \emph{divisorial $(i)$ class} on $X_{n,s}$  as an \textit{effective and  irreducible} divisor of the form \eqref{eq:pic} satisfying following numerical conditions 
\begin{equation}\label{eq:num}
\begin{split}
	& (n-1)d^2-\sum_{j=1}^{s} m_j^2=i, \\
	& d(n+1)-\sum_{j=1}^s m_j=2+i.
\end{split}
\end{equation}

\item We also define {\it $(i)$ Weyl divisors} to be Weyl group orbit of a general hyperplane passing through $n-1-i$ points.
\end{enumerate}
\end{definition}

\begin{remark}
	We first remark that in the planar case the smoothness hypothesis is needed in order to further identify the two concepts above with the arithmetic genus zero via the adjunction formula, see for example \cite{DO}. In higher dimension however, rationality aspects are difficult and are not discussed in this work. So for the simplicity, we will omit the smoothness hypothesis in the definition above.

	Moreover, we remark that the two definitions above, {\it divisorial (i) class} and {\it (i) Weyl divisor}, turn out to be equivalent in any space of the form $X_{n,s}$, but this is not an obvious observation, and it will be proved in Theorem \ref{thm:weyl3}. Namely, it is easy to see that a {\it (i) Weyl divisor} is always a {\it divisorial (i) class} but the converse has to go through more difficult computations. One of the aim of this paper is to exploit the equivalence of the two definitions.
	
As an historical argument, we recall that the equivalence of the two concepts was first observed and used by Nagata in the planar case of blown up $\mathbb{P}^2$ in $s$ points. In the counterexample of Hilbert's 14th problem, Nagata used the equivalence of the $(-1)$ curves, Definition \eqref{intro:-1div} part (1) for $n=2$, to the Weyl group orbit of an exceptional divisor, Definition \eqref{intro:-1div} part (2); these two concepts identify via M. Noether inequality in $\mathbb{P}^2$. He later used the infinity of the Weyl group to conclude the infinite generation of the Cox ring for $X_{2,s}$ for $s=16$, see Section \ref{problem}.

Moreover, we recall that the attempt to prove the equivalence of the two curve classes analogs of Definition \ref{intro:-1div} has failed in $X_{n,s}$, since extra numerical assumptions are needed. It was proved, however in Theorem 6.13 of \cite{dm3} that the two curve analogs recently introduced are equivalent for blown up $\mathbb{P}^3$ in points via the same techniques as developed in this paper. This recent works suggests that the Weyl orbits of curve classes have some numerical interpretation, however the mathematical formulation of the numerical aspect is not obvious.

\end{remark}

Mukai introduced the notion of {\it $(-1)$ divisors} (see Definition~\ref{def:-1mu}). In the planar case, the notion of {\it $(-1)$ divisor} and {\it divisorial $(-1)$ class} is the same with {\it $(-1)$ curve} by Lemma \ref{lem:eq}, but this is not true in higher dimension. 
Throughout the paper we emphasize similarities and differences between the two-dimensional and $n$-dimensional cases, allowing us to conclude that the definition of divisorial $(i)$ classes is a natural generalization of $(i)$ curves. In order to state the main results, let us introduce some notation. We will define these terms more precisely later.  

\begin{notation}\label{not}
Let $\langle, \rangle$ define the Dolgachev-Mukai pairing on $\Pic(X_{n,s})$ (see Equation~\eqref{eq:pair}). By abuse of notation we use the terminology \emph{self-intersection} of a divisor $D$ on $\Pic(X_{n,s})$ to denote $\langle	D,D \rangle$. We also use terminology \textit{anticanonical degree} of the divisor $D$ to denote intersection of $D$ with the anticanonical divisor (see Equation~\eqref{eq:canonical}) rescaled by a factor of $1/(n-1)$, 
\[
\adeg(D):=	\frac{1}{n-1}\langle D,-K_{X_{n,s}}\rangle.
\]
Finally, we use \textit{degree} for the intersection of $D$ with a general hyperplane class $H$ 
\[
\deg(D):=\langle	D,H \rangle
\]
to denote the regular degree of the hypersurface defined by $D$. 

More explicitly, one can see that 
\[
\langle D,D \rangle = (n-1)d^2-\sum_{j=1}^s m_j^2
\quad\text{ and }\quad
\adeg(D)=d(n+1)-\sum_{j=1}^s m_j,
\]
therefore the description of divisorial $(i)$ classes in \eqref{eq:num} can be given as $\langle D, D\rangle=i$ and $\adeg(D)=2+i$.

\end{notation}

The main results of this article are the following three theorems. The first of these theorems, which we prove in Section~\ref{s:Noether}, generalizes the Max Noether inequality for $(-1)$ curves on $X_{2,s}$ to divisorial $(i)$ classes (in higher dimensional spaces, i.e. $n\geq 2$) for $i\in\{-1,0,1\}$. 

\begin{theorem}\label{thm:noether2}
	Let $D$ be a divisor with $d\geq m_j\geq 0$ satisfying $\adeg D=c+2$ and $\langle D, D \rangle=c+e$ for two integers $c$ and $e$ satisfying $-2\leq c, e \leq 1$. When $d=1$ further assume that $\langle D, D \rangle<0$. 
	We have the following: If $s\leq n$, then $m_1+\dots+m_s>(n-1)d$; if $s\geq n+1$, we can reorder the indices so that $m_1\geq m_2\geq \dots\geq m_s$ and $m_1+\dots+m_{n+1}>(n-1)d.$  
\end{theorem}



If $e=0$ then we say that any divisor class satisfying Definition~\ref{def:-1div} (irreducible or not) is not \emph{Cremona reduced} for every $-1\leq i \leq 1$, (see the M. Noether inequality \ref{thm:noether2} of Section \ref{s:Noether}).

One may notice that the assumptions in Theorem~\ref{thm:noether2} are weaker than the assumptions in the original theorem by Max Noether (see e.g. \cite{Dolgachev} and Remark~\ref{rem:Notherhypothesis}). The original Max Noether inequality requires the curve $C$ to be rational and $-2\leq \langle C,C\rangle \leq 1$. This is equivalent to saying $\adeg C=2+\langle C,C\rangle$. From Theorem~\ref{thm:noether2}, however, we deduce that we can apply Max Noether inequality for all $(i)$ curves with $-1\leq i \leq 1$, i.e. fixed curves \emph{and} movable curves.  

The second result, which we prove in Section~\ref{s:Nagatacorrespondence} is a generalization of a result discovered by Nagata in \cite{nagata1} (and reformulated by Dolgachev in \cite{DO}). Namely Nagata's original approach extends to $(i)$ Weyl divisors in $X_{n,s}$ giving  a correspondence between divisorial $(i)$ classes and $(i)$ Weyl divisors, as follows:

\begin{theorem}\label{thm:weyl3} Let $D$ be a divisor in $\Pic(X_{n,s})$. Then $D$ is a divisorial $(i)$ class if and only if it is a $(i)$ Weyl divisor.
In particular, the Weyl group acts transitively on the set of divisorial  $(i)$ classes.
\end{theorem}


Laface and Ugaglia gave an alternate proof only when $i=-1$ in \cite{LUst} via different methods.
However, until now, we did not realize the importance of divisorial $(0)$ divisors and divisorial $(1)$ divisors. 
Indeed, let $\mathcal{Z}_{\geq 0}$
denote the cone of curve classes in $A^{n-1}(X_{n, s})$ that meet all $(0)$-divisorial classes non-negatively.

Further, define $\mathcal{Z}_{\geq 0}\langle -1 \rangle$
to be the cone generated by $\mathcal{Z}_{\geq 0}$ and all $(-1)$-curves in $X_{n, s}$. Theorem~\ref{thm:weyl3} further implies the following Corollary, which also appeared in \cite{dm2}, as an application of this work.

\begin{corollary}
	The cone of classes of effective curves in $X_{n, s}$ is a subcone of 
	$\mathcal{Z}_{\geq 0}\langle -1 \rangle$.
\end{corollary}

As a corollary, we also generalize a result of Dolgachev regarding rational curves with self-intersection $-2$ from the two-dimensional to $n$-dimensional space. 

\begin{corollary}\label{thm:dol} Let $r\in \{-3, -2\}$. There are no irreducible divisors $D$ on $X_{n,s}$ with $ \langle D, D \rangle=r$ and $\adeg(D)=-2-r$.
\end{corollary}

Finally, we want to consider the condition of irreducibility.  
In general, this is a difficult condition to check, and one might ask whether any curve satisfying numerical conditions \eqref{eq:num} is irreducible. 

The answer is no; however, this question motivated the main result in \cite{DO} for planar curves
where it has been shown the irreducibility criterium of Definition~\ref{def:-1div} can be replaced by a numerical condition based on intersection of divisors with smaller degree. 

Moreover, a similar statement is true for divisorial $(i)$ classes in arbitrary dimension. This is done in the following theorem, which we prove in Section~\ref{s:irreduc}.
Furthermore, even in the case of planar curves, Theorem \ref{thm:do2} generalizes the  main result in \cite{DO} from $i=-1$ to $i\in\{-1,0,1\}$.

\begin{theorem}\label{thm:do2} Take $i\in\{-1,0,1\}$. The divisor $D$ is a divisorial $(i)$ class on $X_{n,s}$ if and only if $D$ is effective satisfying numerical conditions \eqref{eq:num} and for all degrees $0<d'<d$ and all divisorial $(-1)$ classes $D'$ of degree $d'$ we have $\langle D, D' \rangle\geq 0$.
\end{theorem}


The remainder of the paper is organized as follows.
In Section~\ref{s:surfaces} we review Cox rings and \emph{$(-1)$ curves} on blown up projective planes.
In Section~\ref{s:weyl} we review standard Cremona transformations, define the Weyl group, and discuss the Dolgachev-Mukai pairing. 
In Section~\ref{s:-1divisors} we introduce \emph{divisorial $(i)$ classes} on blown up projective spaces in $s$ general points. 
As mentioned above, \ref{s:Noether}, Sections~\ref{s:Nagatacorrespondence}, and \ref{s:irreduc} contain the proofs of Theorems \ref{thm:noether2}, \ref{thm:weyl3}, and \ref{thm:do2}, resp. 
Finally, in Section~\ref{s:MDS} we describe divisorial $(-1)$ classes for Mori Dream Spaces and a relationship between $X_{n,n+3}$ and the moduli space of certain rank two vector bundles over $\PP^1$.


\subsection{Recent work.}

Motivated by the birational geometry of $X_{n,s}$, the Weyl group action on cycles of arbitrary dimension was recently introduced and studied following the first drafts of the current article.
Indeed, the current article has served as the \emph{main example and motivation} in recent development of the theory of Weyl actions on the Chow ring $A^*(X_{n,s})$ of blown up projective spaces in $s$ general points. In particular, two definitions of certain Weyl cycles of dimension $r$ in $X_{n,s}$ were introduced in the literature. 

First, the concept of \emph{$r$-Weyl planes}, was introduced in \cite{dm1} as the Weyl group orbit on linear spaces of dimension $r$ passing through $r+1$ fixed points, while the notion of {\it $r$-Weyl cycles} was introduced in \cite{bdp3} as irreducible components of the intersection of pairwise orthogonal of divisorial $(-1)$ classes as defined in the current article (see also \cite{bdp4}). The orthogonality is considered with respect to the Dolgachev-Mukai bilinear form on the Picard group $\Pic(X_{n,s})$. 

For example, \emph{Weyl surfaces} in $X_{4,8}$, the first nontrivial case of  Weyl cycles, were given two different definitions in \cite{dm1} and \cite{bdp4}.  
In \cite{dm1}, they were first defined  via the Weyl group actions on $2$-planes in $X_{4,8}$ and in \cite{bdp4} as irreducible components of the intersections of orthogonal $(-1)$ divisorial classes. The first approach is geometric but difficult to compute, while the second is algebraic but easier to manipulate. 

It turns out the two definitions of Weyl surfaces in $X_{4,8}$ are equivalent and match the results of Casagrande, et al. in \cite{Cas2} obtained by analyzing the birational geometry of the space $X_{4,8}$. In recent work, in \cite{bdps2} the authors prove that the two concepts of $r$-Weyl cycles coincide in Mori Dream Space of type $X^n_s$. In Section~\ref{s:MDS}, we will show an explicit example of computation of a Weyl surface as intersection of two  divisorial $(-1)$ classes.

The theory of Coxeter groups induces a bilinear form on both the space of divisors and the space of curves on $X_{n,s}$, i.e. on $A^{1}(X_{n,s})$ and on $A^{n-1}(X_{n, s})$ (see \cite{dm3}). In the planer case both bilinear forms coincide with the intersection product on the Picard group, $\Pic(X_{2,s})$. The  bilinear form induced by the Coxeter theory on the Picard group $\Pic(X_{n,s})$ is known in the literature as the Dolgachev-Mukai form (see Notation~\ref{not}).

In the planar case of $X_{2,s}$, Nagata's counterexample on Hilbert's 14th problem motivated his work on the equivalence between $(-1)$ curves in $X_{2,s}$ and $(-1)$ Weyl lines (the planar analogue of $(-1)$-divisorial classes). The current article further generalizes Nagata's correspondence between $(-1)$ curves and $(-1)$ Weyl lines, from the planar case of $X_{2,s}$ to arbitrary dimension $X_{n,s}$. For this, the standard intersection product on the Picard group $\Pic(X_{2,s})$ is replaced by the Dolgachev-Mukai bilinear form on $\Pic(X_{n,s})$ induced by the Coxeter theory. More explicitly, the current article shows that divisorial $(i)$ classes in $\Pic(X_{n,s})$ can be equivalently defined numerically---by a linear and a quadratic invariant induced by Dolgachev-Mukai bilinear form.

On the other hand, the recently developed theory of curves in $X_{n,s}$, closely resembles the theory of divisors exploited in this work, offering a complementary approach in the understanding of the birational geometry of $X_{n,s}$. In particular, Weyl actions on curves were first introduced and computed  in \cite{dm2} and \cite{dm3}.

In \cite{dm2}, it was proved that in Mori Dream Spaces of type $X_{n,s}$, the $(-1)$ \emph{Weyl line} classes can also be defined \emph{numerically} via a linear and a quadratic invariant, coming from self-intersection of the curve class and the intersection of curve class with the anti-canonical curve class via the bilinear form of the Coxeter theory on $A^{n-1}(X_{n,s})$. It turns out that $(0)$ and $(1)$ Weyl \emph{lines} give extremal rays for the cone of movable curves, thus defining faces of the effective cone of divisors for Mori Dream Spaces $X_{n,s}$, (see \cite[Theorem 1.4]{dm2}). Moreover, the infinity of $(1)$ Weyl lines reproves the result of Mukai that $X_{n,s}$ is not  a Mori Dream Space whenever $\langle K_{X_{n,s}}, K_{X_{n,s}}\rangle\leq 0$.

In \cite[Theorem 6.11]{dm2} the authors use Theorem \ref{thm:weyl3} to prove that a $(-1)$ Weyl line and a divisorial $(-1)$ class that are part of a base locus of a general effective divisor $D$ don't intersect. This allows one to create resolutions of singularities of an effective divisor $D$ on $X_{n, s}$, and to prove a Riemann Roch type statement (see \cite[Corollary 6.18]{dm2}) for the Euler characteristics of the line bundle $\mathcal{O}(\tilde{D})$; here $\tilde{D}$ stands for the proper transform of $D$ under the blow up of $(-1)$ curves and linear base locus. 

Furthermore, in \cite[Theorem 6.11]{dm2}, the authors use Theorem \ref{thm:weyl3}  to prove that two divisorial $(-1)$ classes contained in the base locus of an effective divisor $D$ have to be orthogonal with respect to the Mukai pairing. This remark motivates the definition of \emph{Weyl cycles} (see \cite{bdp4}) in $X_{n,s}$ as irreducible components of the intersection of $(-1)$ Weyl divisors on $\Pic(X_{n,s})$ that are pairwise orthogonal with respect to the Dolgachev-Mukai pairing.

It was recently discovered that the birational geometry of Mori dream spaces of type $X_{n,s}$ is completely determined by \emph{Weyl cycles} and the geometric computations of such varieties are described in \cite{bdps1, bdps2}. In particular, for Mori Dream spaces, the Mori chamber decompositions of the effective cone of divisors in $X_{n,s}$, originally computed by Mukai via the birational geometry of moduli spaces, is induced by the Weyl actions on $X_{n,s}$. Moreover, in \cite{bdps1} the authors prove that the stable base locus decomposition, the Mori chamber decomposition for $X_{n,s}$ and the Weyl chamber decomposition coincide for Mori Dream spaces of type $X_{n_s}$.

Most importantly, when $X_{n,s}$ is not a Mori dream space, the identification of its birational geometry is a difficult problem. The work in \cite{bdps1, bdps2} opens new directions of investigation, by the identification of countably many walls determined by the $r$ Weyl cycles. Although the pseudo-effective divisorial cone for $X_{n,s}$ may be difficult to analyze in full generality, it is work in progress to identify a structure theorem similar to the structure theorem for the Mori cone of curves for $X_{n,s}$. In this direction, it is important to mention that the recent work of \cite{sx} uses results of the current article to prove that the pseudo-effective divisorial cone of $X_{3,8}$ is closed, and its extremal rays---the divisorial $(-1)$ classes---accumulate to the anti-canonical divisor $-K_{X_{3,8}}$. The same structure theorem was proved to hold for $X_{5,9}$ using results of the current article---most importantly Theorem~\ref{thm:do2} (see also \cite{ bdps2})---as well as results from \cite{sx}.

Finally, let us mention two more results related to the current work. In \cite{Mukai}, Mukai extended the action of Nagata to products of projective spaces. Let $X_{a,b,c}$ denote the blown up product $\PP^{c-1}\times \ldots \times\PP^{c-1}$ of $a-1$ terms, at $b+c$ general points. Then $\Cox(X_{a,b,c})$ is finitely generated if and only if $\frac{1}{a} + \frac{1}{b}+\frac{1}{c}>1$ (see \cite{CT}, \cite{Mu}). 

\begin{question}\label{quest}
	Can the $(i)$-Weyl divisors or $(i)$-Weyl lines be equivalently characterized numerically  via a linear and quadratic invariant on blown up products of projective spaces of type $X_{a, b, c}$?
\end{question}

Theorem 6.13 of \cite{dm3} proves that Question \ref{quest} holds for  $(i)$-Weyl lines on $X_{3,s}$ (corresponding to case $a=2$, $c=n+1$ and $b=s-n-1$ of $X_{a, b, c}$).

In \cite{totaro} Totaro, discusses Hilbert's fourteenth problem over arbitrary fields, in particular over fields of positive characteristics. For some cases, he relates finite generation of the total coordinate ring, to finiteness of a Mordell--Weil group.

\subsection*{Acknowledgements} The work of Professor Wolfgang Ebeling served as an inspiration for the students and postdoctorands in Hannover, in all directions of Algebraic Geometry especially in Singularity Theory, Birational Geometry or Gromov-Witten Theory. The discussions leading to this paper were initiated by the authors during their postdoctoral time in Hannover, in one of the research seminars on Algebraic Geometry and Complex Geometry that they organized. We kindly dedicate this work to the memory of Wolfgang Ebeling.

The collaboration that resulted in this article was first started at two conferences: ``Crossing Walls in Enumerative Geometry'' held at Snowbird Center in Salt Lake City and ``the Geometry and Physics of Quantum Curves'' held at the BIRS center in Banff. The authors would like to thank the organizers for the inspiring conferences. In addition, the authors are also indebted to BYU and CMU for hosting them. Much of the work for this article was accomplished while visiting these two institutions. The authors would also like to thank MPIM Bonn, IHES Bures-sur-Yvette, RIMS Kyoto, IMAR Bucharest, MPI Leipzig for their hospitality during the long visit of 2023, 2024, 2025 and 2026.

This collaboration was partially supported by the National Science Foundation Grant DMS - 1802082 and DMS 2152130 and by the Collaboration Grants 586691 and  855897 from the Simons Foundation.

\section{On Cox Rings and $(-1)$ curves.}\label{s:surfaces} 

In this section we will discuss what is already know about Cox rings and $(-1)$ curves on rational surfaces. We will begin with Cox rings.

\subsection{Cox Rings}\label{cox}
Recall that if $X$ a projective variety whose Picard group is freely generated by divisors $D_1, \ldots, D_r$, then the Cox ring is defined by 
\[
\Cox(X)=\bigoplus_{(n_1,\ldots, n_r)\in \mathbb{Z}^r}H^0(X, n_1D_1+\ldots+n_rD_r).
\]

Notice $X$ is toric if and only if its Cox ring is polynomial. In \cite{HuKeel}, Hu and Keel proved that a projective variety $X$ is a Mori Dream Space if and only if  $\Cox(X)$ is finitely generated. For example, all del Pezzo surfaces (i.e. the blow up of a projective plane in $s$ points, for $s\leq 8$) are Mori Dream Spaces. In fact  in \cite{TVV} the Cox ring of del Pezzo surfaces is proved to be a quadratic algebra---that is the Cox ring is generated in degree 1 with all relations in degree 2. As further examples, it is known that $X_{3,7}$ and 
$X_{4,8}$ are Mori dream spaces, but the Cox ring is only known for $X_{3,7}$ (see \cite{park}). 

Returning to the discussion in Section \ref{intro}, if we consider $G=\mathbb{G}_{a}^{2}$, then Castravet--Tevelev prove in \cite{CT} that $S^G$ is finitely generated, i.e. $X_{s-3,s}$ is a Mori dream space. 
In Nagata's example, we see that $X_{2,16}$ is not a Mori Dream Space, and via Mukai, the same is true for $X_{5,9}$. 

The connection between Hilbert's fourteenth problem and the blow up of $\PP^n$ is realized via the Cox ring. As mentioned in the Introduction, the example of Nagata 
was generalized 
by Steinberg in \cite{st} for $G=\mathbb{G}_{a}^{6}$, 
again by Mukai in \cite{Mukai} for $G=\mathbb{G}_{a}^{3}$ and blowing up $\PP^5$ in nine points (or more generally for $\dim G=i\geq 3$ and blowing up $s\geq \frac{i^2}{i-2}$ points).

An important class of divisors on $X_{2,s}$ are known as \emph{$(-1)$ curves}. Their importance is illustrated in the following results. 
In \cite{BP} Batyrev and Popov proved that the Cox rings for del Pezzo surfaces are generated by global sections of divisor classes that are $(-1)$ curves for $s\leq 7$, and by $(-1)$ curves together with the anticanonical class for $s=8$.  Moreover, $\Cox(X_{n,s})$ is generated by global sections of  \emph{divisorial $(-1)$ classes} for $s\leq n+3$ (see e.g. \cite{CT}). 

But even when $X_{n,s}$ is not a Mori Dream Space, the $(-1)$-curves play an important role in the geometry of the space. In this paper we prove several results generalizing what is known about $(-1)$ curves (in dimension two) to higher dimension. But first, let us review what is known about $(-1)$-curves.

\subsection{$(-1)$ curves on rational surfaces.}

The theory for $n=2$ is particularly nice. In this section we recall main important results obtained for  $X:=X_{2,s}$ the blown up projective plane in a collection of $s$ general points. The main results of the later sections are in large part inspired by what we know in dimension two, with some important differences, which we will try to point out. 
In this case a divisor is a planar curve. 

We recall the usual intersection pairing on $\Pic(X)$
\[
\cdot: \Pic X\times \Pic X\to \ZZ;
\]
recall that 
\begin{equation}
\begin{split}
H \cdot H &= 1,\\
H\cdot E_i&= 0,\\
E_i \cdot E_j &= -\delta_{i,j}.
\end{split}
\end{equation}

We recall the following definition of $(-1)$ curves on $X$ (see also \cite{nagata1}).
\begin{definition}\label{def:-1 curve} A smooth divisor $D\in \Pic X$ is a \textit{$(-1)$ curve} if $D$ is  irreducible, rational and has self-intersection $D \cdot D=-1$.
\end{definition}

To illustrate the importance of $(-1)$ curves on $X$ we will first recall two important conjectures in the Interpolation Problems area. 

The first of these is a conjecture by Gimigliano-Harbourne-Hirschowitz regarding effective divisors on $X$ (see e.g. \cite{Gim89, Har86, Hir89}). Ciliberto and Miranda proved in \cite{cm} that this is also equivalent to a conjecture of Segre; nowadays this conjecture is known as SGHH conjecture. 

Let $\chi(X, \mathcal{O}_X(D))$ denote the Euler characteristic of the sheaf $\mathcal{O}_X(D)$. This can be computed via Riemann-Roch theorem for divisors on the rational surface $X$.

\begin{align*}
\chi(X, \mathcal{O}_X(D)) &=1+\frac{D\cdot (D-K_X)}{2}\\
 &=\dim H^0(X,\mathcal{O}_X(D))- \dim H^1(X,\mathcal{O}_X(D)) + \dim H^2(X,\mathcal{O}_X(D)).
\end{align*}

\begin{remark}
We further remark that for divisors $D$ of $\Pic(X_{n,s})$ with positive coefficients all the higher cohomology groups vanish, namely $H^i(X,\mathcal{O}_X(D))=0$ for $i\geq 2$. This statement can be easily proved for each such divisor by induction on the multiplicity. Therefore, only nontrivial cohomology for such divisors can arise from $H^0(X,\mathcal{O}_X(D))$ and $H^1(X,\mathcal{O}_X(D))$.
\end{remark}
\begin{conjecture}[Gimigliano-Harbourne-Hirschowitz]\label{conj:GHH} Let $D\in \Pic (X)$ an effective divisor of form \eqref{eq:pic}. Then 
	\begin{equation}
	\chi(X, \mathcal{O}_X(D))=\dim H^0(X,\mathcal{O}_X(D))
	\end{equation}
	if and only if $D \cdot C\geq -1$ for all $(-1)$ curves $C$ on $X$.
\end{conjecture}

In other words, the Euler characteristic is closely connected with the pairing against $(-1)$-curves. 

\begin{remark}\label{rem:ghhr}
	One can formulate Conjecture~\ref{conj:GHH} for any general effective divisor $D$ as follows: Let $C$ be any $(-1)$ curve so that $k_C:= - D \cdot C>0$, then
\[
\dim H^0(X,\mathcal{O}_X(D))= \chi(X, \mathcal{O}_X(D)) + \sum_{C \cdot D<0} \binom{k_C}{2}.
\]
\end{remark}

In order to state the second conjecture, consider a homogeneous divisor $D$, i.e. a divisor that has all multiplicities equal $m_1=\ldots=m_s=m$. For degree and multiplicities large enough, whenever $\frac{d}{m}<\sqrt{s}$, then $\chi(X, \mathcal{O}_X(D))=\tfrac{1}{2}(d^2-sm^2)+\tfrac{3d-sm+2}{2}$ is negative. By symmetry, if $s> 9$ then no $(-1)$ curve can intersect a homogeneous divisor $D$ negatively and by Conjecture~\ref{conj:GHH} one expects that $\dim H^0(X,\mathcal{O}_X(D))=0$. This motivates the following conjecture.

\begin{conjecture}[Nagata]\label{conj:nag} 
	If $\frac{d}{m}<\sqrt{s}$ then $D$ is not effective.
\end{conjecture}

Nagata proved this conjecture 
when $s$ is a perfect square \cite{nagata1}. In fact, both Conjectures~\ref{conj:GHH} and \ref{conj:nag} are known to be true for homogeneous divisors whenever the number of points, $s$, is a perfect square \cite{cm2}. Conjecture~\ref{conj:nag} also has a more general form for non-homogeneous divisors, but for simplicity we will not discuss it here. 

These conjectures illustrate the importance of $(-1)$ curves on $X$, and we
now describe what is known about $(-1)$ curves.  
In later sections, we will obtain similar results to those mentioned here but in higher dimension.

Let $W_{2,s}$ denote the Weyl group on $X_{2,s}$ generated by reflections (see Section~\ref{ss:weylgroup} for more discussion of this group.)
In 1960 Nagata introduces classes of \emph{pre-exceptional type}, namely the orbit of the Weyl group $W_{2,s}$ action on one exceptional divisor. 
He proves the fundamental result that they are in bijective correspondence with $(-1)$ curves as in the following theorem:

\begin{theorem}[Nagata, \cite{nagata1}]\label{thm:weyl} There exists a bijection between the set of $(-1)$ curves on $X$ and the orbit of the Weyl group on one exceptional divisor, say $W_{2,s} \cdot E_i$. In particular the Weyl group acts transitively on the set of exceptional curves.
\end{theorem}

Theorem~\ref{thm:weyl} of Nagata is based on the Lemma of Max Noether which is further exposed in \cite{Dolgachev} (and slightly reformulated in \cite[Lemma 2.2]{DO}). We obtain a similar result for $n\geq 2$ (see Theorem~\ref{thm:weyl3} and the proof in Section~\ref{s:Nagatacorrespondence}).

\begin{theorem}[M. Noether's inequality]\label{thm:noether1} Let $D$ be the class of an irreducible rational curve satisfying $-2\leq D \cdot D\leq 1$. Then there exist $i_1<i_2<i_3$ such that 
\[
m_{i_1}+m_{i_2}+m_{i_3}>d.
\]
\end{theorem}

The following result (see e.g. \cite[Proposition 2.1]{DO}) provides an alternative definition of $(-1)$ curves. This result was one of the main motivations for the current work. Notice that the irreducibility assumption is not needed for next statement. Let $p_a(D)$ denote the arithmetic genus of the divisor $D$.

\begin{lemma}\label{lem:eq} 
Let $D=dH-\displaystyle\sum_{i=1}^{s} m_i E_i$ be an arbitrary curve class on $X$. Then any two conditions imply other two:
	
\begin{enumerate}
	\item\label{item:-1intersection} $D \cdot D  =-1$. 
	\item $p_a(D)=0$ (i.e. $D$ is a rational curve).
	\item\label{item:-1curveAdeg} $D \cdot (-K_X) =1$.
	\item\label{item:-1euler} $\chi(X, \mathcal{O}_X(D)) = 1$.
\end{enumerate}
\end{lemma}

The proof of this result follows from the adjunction formula 
\begin{equation}\label{eq:agenus}
p_a(D)=\frac{2+D\cdot (D+K_X)}{2}
\end{equation}
and Riemann-Roch theorem for divisors on the rational surface $X$. Notice that the first two conditions together with irreducibility of the divisor define $(-1)$ curves. This leads to the following corollary.

\begin{corollary} A divisor $D\in \Pic X$ is a \textit{$(-1)$ curve} if $D$ is smooth, irreducible on $X$ and any two conditions of Lemma~\ref{lem:eq} hold.
\end{corollary}

Unfortunately, neither of these last two results is true for $n>2$. It is important to make the following remark

\begin{remark}\label{rem:eff} Any divisor $D$ on the blown up plane $X_{2,s}$, satisfying any two conditions of the set of four conditions of Lemma~\ref{lem:eq} is effective. In particular $(-1)$ curves are effective.
This follows from equality (4) since 
\[
\chi(X, \mathcal{O}_{X}(D))=\dim H^0(X, \mathcal{O}_{X}(D))-\dim H^1(X, \mathcal{O}_{X}(D))=1.
\]
\end{remark}

Let us also mention that although effectivity of $(-1)$ divisors is implied by the definition in dimension two, this is no longer the case in higher dimension. In Example~\ref{eg:NotEffective} we see a divisor satisfying the numerical conditions required for being a divisorial $(-1)$ class, but the curve is not effective, so it won't be in the Weyl group action $W_{n,s}\cdot E_i$. Therefore one needs to introduce effectivity in the definition of divisorial $(i)$ classes!

The following result, stronger than Remark~\ref{rem:eff}, holds only in dimension two. It follows from by property (3) of Corollary~\ref{cor:cre} and property (4) of Lemma~\ref{lem:eq}.

\begin{corollary} If $D$ is a \textit{$(-1)$ curve} then 
	\begin{align*} 
	& \dim H^0(X,\mathcal{O}_X(D))=1\\
	& \dim H^1(X,\mathcal{O}_X(D))=0.
	\end{align*}
	In particular Conjecture~\ref{conj:GHH} holds for $D$.
\end{corollary}

In addition, Dolgachev obtains the following result, using the same techniques as Nagata.

\begin{proposition}[Dolgachev \cite{Dolgachev}]\label{prop:dolg1} There are no irreducible, rational curves on $X$ with $D\cdot D=-2$.
\end{proposition}

In \cite{DO} the authors prove that the irreducibility condition for $(-1)$ curves can be replaced by a numerical condition. More precisely, they prove the following:

\begin{theorem}[Dumitrescu--Osserman, \cite{DO}]\label{thm:do} Let $X$ be the blow up of $\mathbb{P}^2$ at very general points $p_1,\ldots, p_s$. A divisor class $D$ is the class of a $(-1)$ curve if and only if either it is one of $E_i$ or it is of the form $dH-m_1E_1-\ldots - m_sE_s$ with $d>0, m_i\geq 0$ for all $i$ so that any two equivalent conditions of Lemma~\ref{lem:eq} hold and moreover for all $0< d'<d$ and all $(-1)$ curves $C$ of degree $d'$ on $X$, we have $D \cdot C\geq 0$.
\end{theorem}

\begin{example}\label{eg:ex1}
Notice that the last condition is needed. Indeed, as observed by Dumitrescu--Osserman in \cite{DO},	the divisor $D=5H-3E_1-3E_2-E_3-\ldots-E_{10}$ satisfies the numerical conditions \eqref{item:-1intersection} and \eqref{item:-1curveAdeg} in Lemma~\ref{lem:eq}, but fails the last condition of Theorem~\ref{thm:do}, with $C=H-E_1-E_2$. We see that $C$ is a $(-1)$ curve, however $D\cdot C=-1$.  
Moreover, we see that $D$ is not irreducible as $D$ it splits as the sum of two curves $H-E_1-E_2$ and $4H-2E_1-2E_2-E_3-\ldots-E_{10}$. Therefore $D$ is not a $(-1)$ curve. 
\end{example}

\begin{example}
Theorem~\ref{thm:do} essentially says that if $D$ satisfies the numerical conditions of Theorem~\ref{lem:eq}, then we can check irreducibility by intersecting with $(-1)$ curves of smaller degree. For general curves in the blow up $X_{2,s}$, irreducibility cannot be tested by intersection with $(-1)$ curves. The theorem only applies to divisors satisfying the conditions of Lemma~\ref{lem:eq}. Indeed, take a sextic with nine double points $D=6H-2E_1-\ldots-2E_9$ and notice that this curve does not satisfy the requirements of $(-1)$ curve, since $D\cdot D\neq -1$, and it does not have arithmetic genus $p_a(D)\neq 0$. Furthermore, $D$ is not irreducible; it consists of the double cubic passing through the nine points. However, one can check that $D$ does satisfy the numerical criterion $D \cdot C\geq 0$ for all $(-1)$ curves $C$ as in Theorem~\ref{thm:do}. 
\end{example}

\section{Weyl group action on $\Pic(X_{n,s})$ and the Mukai pairing.}\label{s:weyl}

In this section we discuss the Weyl group action and an important class of strongly birational maps on $\PP^n$, called the standard Cremona transformations. 

\subsection{Weyl group action}
The \emph{standard Cremona transformation based at the $n+1$ coordinate points} of $\PP^{n}$ is defined to be the birational map
\[
[x_0, \dots, x_{n}]\rightarrow \bigg[\frac{1}{x_0}, \dots, \frac{1}{x_{n}}\bigg].
\]
This map is given by divisors of degree $n$ with multiplicity $n-1$ at each of the $n+1$ coordinate points.  
The standard Cremona transformation contracts each of the coordinate hyperplanes to a point.
The indeterminacy locus of the standard Cremona transformation consists of the collection of $n+1$ coordinate points, and all linear subspaces of dimension at most  $n-2$ generated by these points. If we lift to $X_{n,n+1}$, we get a strong birational map on $X_{n,n+1}\dashrightarrow X_{n,n+1}$, i.e. an isomorphism in codimension 1.

Moreover it induces an automorphism of the Picard group of $X_{n,s}$
for $s\geq n+1$ points and by abuse of notation we will denote $\Cr: \Pic X_{n,s}\stackrel{}{\rightarrow} \Pic X_{n,s}$
via the rule 
\begin{equation}\label{eq:crem}
\Cr(dH-\sum_{i=1}^{s}m_{i}E_{i})= 
(d-k)H-\sum_{i=1}^{n+1}(m_i-k)E_{i}-\sum_{i=n+2}^{s}m_{i}E_{i} 
\end{equation}
where 
\[
k=m_{1}+\dots+m_{n+1}-(n-1)d
\] 
and the first $n+1$ points are chosen to be the coordinate points of $\PP^{n}$.

Denote the canonical divisor on $X_{n,s}$ by
\begin{equation}\label{eq:canonical}
K_{X_{n,s}}:= -(n+1)H+(n-1)E_1+\dots +(n-1)E_s.
\end{equation}

\begin{remark}\label{rem:cr} The standard Cremona transformation of $\mathbb{P}^n$ (1) is an involution of $\mathbb{P}^n$, that (2) fixes canonical divisor of $X_{n,s}$, and (3) preserves semigroup of effective divisors. Moreover, (4) Cremona transformation preserves dimension of space of global sections of divisors (see \cite{Dumnicki}). These four points are summed up in the following equations. 

\begin{enumerate}
\item\label{item:CrInvolution} $\Cr \Cr D = D$
\item\label{item:CrCanon} $\Cr K_{X_{n,s}}=K_{X_{n,s}}$.
\item\label{item:CrEffective} If $D\geq 0$ then $\Cr D\geq 0$.
\item \label{item:globalsections} $\dim H^0(X_{n,s}, \mathcal{O}(D))=\dim H^0(X_{n,s}, \mathcal{O}(\Cr D))$.
\end{enumerate}
\end{remark}

\begin{remark}\label{rem:ex1} 
By property~\eqref{item:globalsections} of Remark~\ref{rem:cr}, if $D$ is an effective divisor then $\Cr(D)$ is also effective so it has positive degree. 
Indeed, one can also check that \[
\deg\Cr(D)=nd-\displaystyle\sum_{i=1}^{n+1} m_i>0
\]
since $D$ is effective. However, the multiplicities of $\Cr(D)$ may not all be positive. Indeed, 
\[
m_{n+1}-k=(n-1)d-\displaystyle\sum_{i=1}^{n}m_i
\]
can be negative if $D$ is not irreducible---for example, if the hyperplane through first $n$ points is a fixed component of $D$.
\end{remark}

\begin{remark}	
	In \cite{Dolgachev} Dolgachev defines a Cremona isometry to be an automorphism of $A^{1}(X)$ preserving Dolgachev-Mukai intersection pairing (defined in Equation~\eqref{eq:pair}, see also Theorem~\ref{thm:cc})  and Properties~\eqref{item:CrCanon} and \eqref{item:CrEffective} of Remark~\ref{rem:cr}. He further proves that the group of effective Cremona isometries---the ones induced by automorphisms of $X$---is the Weyl group. 
	This property is far from being true if $\dim X\geq 3$.
\end{remark}

We can generalize the map $\Cr$ to include any subset $I\subset \{1,2,\dots , s\}$ of size $n+1$ by precomposing $\Cr$ with a projective transformation, taking the points indexed by $I$ to the $n+1$ coordinate points of $\PP^n$. This transformation is also called a \emph{standard Cremona transformation}, and we denote it by $\Cr_I$. In other words, a standard Cremona transformation is a transformation projectively equivalent to $\Cr$. Obviously, the properties of Remark~\ref{rem:cr} also hold for $\Cr_I$.

For later section it is useful to mention the following result that holds only for blown up planes.

\begin{theorem}\label{thm:coh} For any divisor $D$ on $X=X_{2,s}$ we have the following
\begin{align*} 
& \dim H^0(X, \mathcal{O}(\Cr D))=\dim H^0(X, \mathcal{O}(D))\\
& \dim H^1(X, \mathcal{O}(\Cr D))=\dim H^1(X, \mathcal{O}(D)).
\end{align*}
\end{theorem}

\begin{proof}
To prove this result apply Riemann--Roch formula in dimension 2:
\[
\chi(X, \mathcal{O}_X(\Cr D))=1+\frac{\Cr D\cdot (\Cr D-K_X)}{2}=\chi(X, \mathcal{O}_X(D)).
\]
Indeed we will see in Theorem~\ref{thm:cc} that $\Cr$ preserves the intersection pairing (or the Dolgachev-Mukai pairing in higher dimension). Conclude with Properties~\eqref{item:CrCanon} and \eqref{item:globalsections} of Remark~\ref{rem:cr}. 
\end{proof}

\subsection{Root systems and Weyl groups.}\label{ss:weylgroup}
The exposition of this section  follows Dolgachev (see \cite{Dolgachev}) and Mukai (see \cite{Mukai} and \cite{Mu}).
In \cite{manin}, Manin associated the group $E_6=T_{3,2,2}$ to the configuration of 27 lines on a nonsingular cubic surface in $\PP^3$---i.e. the blow up of $\PP^2$ at six points. This result was generalized by Dolgachev for $X_{n,s}$ as we describe below.
For $s\geq n+1$ let $L$ be a lattice of rank $s+1$ with orthogonal basis $H, E_1, \ldots, E_s$. The orthogonal complement of the canonical divisor $K_{X_{n,s}}$ (see Equation~\eqref{eq:canonical}) has basis $\mathcal{B}$ given by
\begin{align*}
\alpha_1:=E_1-E_2, \quad \alpha_2&:=E_2-E_3, \quad\ldots\quad \alpha_{s-1}:=E_{s-1}-E_{s}, \text{ and }\\ 
\alpha_{s}&:=H-\sum_{i=1}^{n+1}E_i
\end{align*}
that becomes a root system for the vector space $V=\Pic(X_{n,s})\otimes_{\mathbb{Z}} \mathbb{R}$. The dual base is $\mathcal{B}^{\vee}=\{\alpha_1^{\vee}, \ldots, \alpha_{s}^{\vee}\}$ in $N_1(X_{n,s})$,  where $\alpha_i^{\vee}=f_i-f_{i+1}$ for $i\leq s-1$ and 
$\alpha_s^{\vee}=(n-1)l-\displaystyle\sum_{i=1}^{n+1}f_i$. 
Here, $f_i$ denotes the class of a line in the exceptional divisor $E_i$ and $l$ a general line class on $X_{n,s}$, so $f_i\cdot E_i=-1$ and $l\cdot H=1$.
Let $T_i:V\stackrel{}{\rightarrow}V$ be the simple reflections for $1\leq i\leq s$ defined by
\[
T_i(x):=x+\alpha_i^{\vee}(x)\cdot \alpha_i.
\]

For any $i<s$, we see that $T_i(E_{i})=E_{i+1}$ and $T_i(E_{i+1})=E_{i}$ and $T_i$ leaves the other bases elements of $\Pic(X_{n,s})$
fixed, while
\begin{equation}\label{eq:weylgenerators}
\begin{split}
T_s(H)&=nH-(n-1)\sum_{i=1}^{n+1} E_i\\ 
T_s(E_j)&=H-\sum_{i\neq j, i=1}^{n+1} E_{i} \quad \text{for } j\leq n+1\\
T_s(E_j)&=E_j \qquad\text{for } j>n+1
\end{split}
\end{equation}

From this description, we can recognize $T_s$ as the automorphism induced by $\Cr$ on $\Pic (X_{n,s})$ as described in \eqref{eq:crem}.
The Dynkin diagram of the group generated by the $T_i$ for $1\leq i\leq s$ is often described as $T_{n+1, s-n-1, 2}$, which denotes a T-shaped graph with three legs of length $2$, $n+1$ and $s-n-1$, resp.

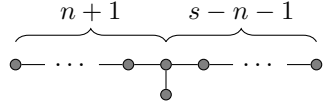
\begin{figure}[h]
	\centering
	\begin{tikzpicture}[xscale=.5,yscale=.4]

\begin{scope}[every node/.style={circle, draw, fill=black!50, inner sep=0pt, minimum width=4pt}]
	\node (n1) at (-4,0) {};
	\node (n3) at (-1,0) {};
	\node (n4) at (0,0) {};
	\node (n5) at (0,-1) {};

	\node (n8) at (1,0) {};
	\node (n10) at (4,0) {};
\end{scope}

	\node (n9) at (2.5,0) {$\dots$};
\node (n2) at (-2.5,0) {$\dots$};
	
	\foreach \from/\to in {n1/n2,n2/n3,n3/n4,n4/n5,n4/n8,n8/n9,n9/n10}
	\draw (\from) -- (\to);

	\draw [decorate,decoration={brace,amplitude=5pt,raise=2ex}]
	(-4,0) -- (0,0) node[midway,yshift=2em]{$n+1$};
		\draw [decorate,decoration={brace,amplitude=5pt,raise=2ex}]
	(0,0) -- (4,0) node[midway,yshift=2em]{$s-n-1$};
	\end{tikzpicture}
	\caption{Dynkin diagram for $T_{n+1,s-n-1,2}$}
\end{figure}

The construction of Dolgachev was generalized by Mukai in \cite{Mu} for products of projective spaces $X_{a,b,c}$ whose corresponding root systems are comprised of $a+b+c-2$ vertices representing a basis for the vector space $\Pic(X_{a,b,c})\otimes_{\mathbb{Z}} \mathbb{R}$. The Dynkin diagram in this case is $T_{a,b,c}$, which has the shape of  a ``T'' and with the three legs having length $a$, $b$ and $c$, resp.

\begin{definition}
The \emph{Weyl group} $W_{n,s}$  is defined to be the group generated by all simple reflections $T_i$ on $X_{n,s}$ where $1\leq i\leq s$.
\end{definition}

\begin{remark}\label{rem:composition}
Any element of the Weyl group $w\in W_{n,s}$ is a composition of standard Cremona transformations based at arbitrary subsets of $n+1$ points of $\{1,\ldots, s\}$. In other words, there exist index subsets $I_1,\ldots, I_t$, where $I_j\subset \{1,\ldots, s\}$ and $|I_j|=n+1$, so that
\[
w=\Cr_{I_1}\circ \Cr_{I_2} \circ \ldots \circ \Cr_{I_t}.
\]

Obviously the inverse of $w$ in the Weyl group $W_{n,s}$ is
\[
w^{-1}=\Cr_{I_t}\circ \Cr_{I_{t-1}} \circ \ldots \circ \Cr_{I_1}.
\]
\end{remark}

\begin{corollary}\label{cor:cre}
Properties~\eqref{item:CrCanon}, \eqref{item:CrEffective} and \eqref{item:globalsections} of Remark~\ref{rem:cr} hold for any Weyl group element $w\in W_{n,s}$.
\end{corollary}

The group of \textit{birational automorphisms} of a projective space $\mathbb{P}^n$ is called  Cremona group (see e.g. \cite{Dolgachev}). 




\subsection{Properties of the Dolgachev-Mukai pairing.}\label{s:pairing}

We introduce a pairing on Picard group $\Pic(X_{n,s})$ following \cite{Mu} (recall the description of $\Pic(X_{n,s})$ in \eqref{eq:pic}):
\[
\langle , \rangle: \Pic X_{n,s}\times \Pic X_{n,s}\to \ZZ.
\]
The pairing has a simple description:
\begin{equation}\label{eq:pair}
\begin{split}
\langle H, H\rangle &= n-1,\\
\langle H, E_i\rangle &= 0,\\
\langle E_i, E_j\rangle &= -\delta_{i,j}.
\end{split}
\end{equation}

By B\'ezout theorem for $n=2$, the Dolgachev-Mukai pairing $\langle C, F \rangle=C \cdot F$ coincides with the intersection of two general divisors (curves) $C$ and $D$ on $X$.

\begin{definition}
For a divisor $D=d_1H-\displaystyle\sum_{i=1}^sm_iE_i\in \Pic(X_{n,s})$ denote by
$\widetilde{D}:=d_1H -d_1E_0- \displaystyle\sum_{i=1}^s m_iE_i\in \Pic(X_{n+1,s+1})$ to be the cone over $D$  with vertex at the exceptional divisor denoted by $E_0$. This cone consists by the union of all lines through $E_0$ and points of $D$. Part (1) of next Theorem was also observed by Laface and Ugaglia in \cite{LUst}.
\end{definition}

\begin{theorem}\label{thm:cc} The following two statements hold: 
	
\begin{enumerate} 
\item\label{item:CrIntersection} The Cremona transformation on $X_{n,s}$ preserves the Dolgachev-Mukai pairing and the anticanonical degree of divisors.
\item\label{item:CrCones} Cones in $X_{n+1,s+1}$ over divisors in $X_{n,s}$ with the same vertex set preserve the intersection pairing Dolgachev-Mukai and the anticanonical degree of divisors.
\end{enumerate}
\end{theorem}

\begin{proof} For \eqref{item:CrIntersection}, let $\displaystyle D:=d_1H - \sum_{i=1}^s m_iE_i$ and $\displaystyle F=d_2H-\sum_{j=1}^s p_jE_j$ be divisors on $X_{n,s}$ with $m_i, p_i\geq 0$. 
From our description of the Cremona action \eqref{eq:crem} we see
\begin{align*}
\Cr D&:= (d_1-k_1) H - \sum_{i=1}^{n+1} (m_i-k_1)E_i - \sum_{j=n+2}^s m_j E_j\\
\Cr F&:= (d_2-k_2) H - \sum_{i=1}^{n+1} (p_i-k_2)E_i - \sum_{j=n+2}^s p_j E_j 
\end{align*}
where $k_1:=m_1+\ldots+ m_{n+1}- (n-1)d_1$ and $k_2:=p_1+\ldots+ p_{n+1}- (n-1)d_2$. Then
\begin{align*}
\langle \Cr D, \Cr F \rangle & = (n-1)(d_1-k_1)(d_2-k_2)- \sum_{i=1}^{n+1} (m_i-k_1)(p_1-k_2) - \sum_{j=n+2}^s m_jp_j\\
&= \langle D, F \rangle - k_1[k_2+(n-1)d_2 - \sum_{j=1}^{n+1} p_j]-k_2[k_1+(n-1)d_1-\sum_{i=1}^{n+1} m_i]\\
&= \langle D, F \rangle.
\end{align*}

By Property~\eqref{item:CrCanon} of Remark~\ref{rem:cr}, the canonical divisor $-K_{X_{n,s}}$ is invariant under the Cremona action. (See Equations~\eqref{eq:canonical} and \eqref{eq:crem} for descriptions of the canonical dvisior and $\Cr$, resp.)
\[
\Cr K_{X_{n,s}}=K_{X_{n,s}}.
\]

We conclude
\begin{align*}
\adeg (D): &=\frac{\langle D, - K_{X_{n,s}} \rangle}{n-1} \\
&=\frac{\langle \Cr D, - \Cr K_{X_{n,s}} \rangle}{n-1} \\
&=\frac{\langle \Cr D, -  K_{X_{n,s}} \rangle}{n-1} \\
&=\adeg (\Cr(D)).
\end{align*}

To prove \eqref{item:CrCones}, take $\displaystyle \widetilde{D}:=d_1H -d_1E_0- \sum_{i=1}^s m_iE_i$ and
$\displaystyle\widetilde{F}=d_2H- d_2E_0-\sum_{j=1}^s n_jE_j$ cones in $X_{n+1, s+1}$  over divisors $D$ and $F$. 
\begin{align*}
\langle  \widetilde{D}, \widetilde{F} \rangle & = nd_1d_2- d_1d_2-\sum_{i=1}^{s} m_in_i \\
&=(n-1)d_1d_2- \sum_{i=1}^{s} m_in_i\\
&= \langle D, F \rangle\\
\adeg\widetilde{D} & = (n+2) d_1 - d_1-\sum_{i=1}^s m_i\\
&=(n+1)d_1-\sum_{i=1}^s m_i\\
&=\adeg(D).
\end{align*}
\end{proof}

\section{Divisorial $(i)$ classes on blown up projective space.}\label{s:-1divisors}

In this section we define divisorial $(i)$ classes we will be studying, and prove some preliminary results. We also give examples of what we call sporadic divisors. 

We are now prepared to make the following definition.  

\begin{definition}\label{def:-1div}
 Let $i\in \{-1,0, 1\}$ and $D\in \Pic X_{n,s}$ be a  divisor

\begin{enumerate}
\item We say $D$ is a \emph{divisorial $(i)$ class} on $X_{n,s}$ if $D$ is an effective and irreducible divisor (possibly non-reduced) that satisfies the following two conditions:
\begin{enumerate}
	\item $\langle	D,D \rangle= i,$ \label{item:selfint}
	\item $\adeg(D)=	\frac{1}{n-1}\langle D,-K_{X_{n,s}}\rangle = 2+i.$ \label{item:adeg}
\end{enumerate}
\item We say  $D$ is a {\it $(i)$ Weyl divisor} if there exists $w\in W_{n,s}$ such that $D=w(H_{n-1-i})$, where $H_{n-1-i}$ is a hyperplane passing through $n-1-i$ points.

\end{enumerate}
\end{definition}

For two examples, we will show in Lemma~\ref{lem:exceptional} that $E_j$ is a divisorial $(-1)$ class for $1\leq j \leq s$, and that a general hyperplane $H_{n-1-i}$ passing through $n-1-i$ points is a divisorial $(i)$ class.  

\begin{remark} Conditions \eqref{item:selfint} and \eqref{item:adeg} in the previous definition are a restatement of Equation~\eqref{eq:num} in the Introduction. In \cite[Definition 4.1]{LUst}, Laface and Ugaglia define \emph{$(-1)$ classes} to be divisors $D$ that are effective, \textit{reduced} and irreducible satisfying Conditions \eqref{item:selfint} and \eqref{item:adeg}. This is equivalent to Definition~\ref{def:-1div}, when $i=-1$. However, Lemma~\ref{reduced} proves that the reducibility assumption---which does not appear in this definition---
is redundant. In \cite{park} Park and Lesieutre briefly relate $(-1)$ divisors to effective divisors satisfying conditions \eqref{item:selfint} and \eqref{item:adeg}.  
\end{remark}

\begin{lemma}\label{reduced} Let $i\in \{-1,0,1\}$ and $D\in \Pic X_{n,s}$ be an irreducible divisor 
satisfying $\langle	D,D \rangle= i$ and $\adeg D=2+i$. Then $D$ is reduced.
\end{lemma}

\begin{proof}
Assume $D$ is a non reduced irreducible divisor satisfying $\langle	D,D \rangle= i$. Assume, for some positive integer $m>1$, we have
\[D=mF.\]

We obtain
\[
\langle	D,D \rangle= \langle	mF,mF \rangle= m^2 \cdot\langle	F,F \rangle= i.
\]
For $i\in\{-1,1\}$ we obtain  a contradiction, since $\langle	D,D \rangle$ is an integer and $m>1$.

Assume now $i=0$, then $\langle	D,D \rangle= 0$ implies $\langle	F,F \rangle= 0$. Now $\adeg D=2$ implies that $m\cdot \adeg F=2$; but the condition $m>1$ implies $m=2$,  i.e. $\adeg F=1$. A divisor with these numerical properties cannot exist, by Corollary \ref{coro}. We obtain a contradiction.
\end{proof}

The next example shows that unlike the case of $\PP^2$ (see Remark~\ref{rem:eff}), the numerical conditions of Definition~\ref{def:-1div} are not enough to guarantee effectivity. 
(Also see Definition~\ref{def:-1 curve}.) 
This shows why the effectivity hypothesis is needed in the definition. 

\begin{example}\label{eg:NotEffective} Take $i\in \{-1,0,1\}$ and $D:=10H- 7E_1-6E_2-6E_3 - 6E_4-6E_5-E_6-\ldots-E_{12-i}$ in $\Pic(X_{3, 12-i})$. Then $D$ satisfies numerical conditions of a  divisorial $(i)$ class since
\begin{align*}
&\langle D, D \rangle =(3-1)\cdot 10^2-7^2-4\cdot 6^2-(7-i)=i\\
&\adeg D=\frac{\langle D, - K_{X_{3,13}} \rangle}{2} =4\cdot 10-(38-i)=2+i
\end{align*}

Note that in $\Pic(X_{n,s})$ any effective divisor satisfies condition $\displaystyle nd \geq \sum_{j=1}^{n+2} m_j$ for any index set of size $n+2$ of $\{1,\ldots, s\}$ while $d\geq m_j$ for any $1\leq j\leq s$. We see that $D$ is not effective, since $3\cdot 10<7+6+6+6+6$ (even though $d\geq m_j$).
\end{example}

We also emphasize that Lemma~\ref{lem:eq} doesn't hold on $X_{n,s}$. Notice that the numerical conditions of Definition~\ref{def:-1div} are conditions \eqref{item:-1intersection} and \eqref{item:-1curveAdeg} of Lemma~\ref{lem:eq}. We now present an example of a divisor satisfying numerical conditions of Definition~\ref{def:-1div} (in fact we will see by Theorem~\ref{thm:weyl3} that $D$ is a divisorial $(-1)$ class) that does not satisfy $\chi(X, \mathcal{O}_{X}(D))=1$ (condition~\eqref{item:-1euler} of Lemma~\ref{lem:eq}). Moreover, standard Cremona transformations do not preserve higher cohomology groups, and therefore also do not preserve the Euler characteristics of divisors (as in Theorem~\ref{thm:coh}).

\begin{example}\label{eg:ex4} 
Consider the following divisor in $X=X_{4,7}$ in the exception list of the celebrated Alexander--Hirschowitz Theorem 
\[
D:=3H-2E_1-\ldots-2E_7.
\]

It is easy to see that $D$ satisfies numerical conditions \eqref{eq:num}. In fact, $D$ is in the Weyl group orbit of an exceptional divisor $W_{4,7}\cdot E_3$, 
therefore by Theorem~\ref{thm:weyl3}, it is effective and irreducible so $D$ is a divisorial $(-1)$ class. To see this consider the following sets of indices: $I_1=\{1,2,3,4,5\}$, $I_2=\{1,2,3,6,7\}$ and $I_3=\{3,4,5,6,7\}$. One can check that $D=\Cr_{I_1}\Cr_{I_2}\Cr_{I_3}E_3$. 
Moreover, from Property~\ref{item:globalsections} of Remark~\ref{rem:cr} we see that
\[
\dim H^0(X, \mathcal{O}_{X}(D))=1.
\]

Notice that
\[
\chi(X, \mathcal{O}_{X}(D))=\binom{7}{4}-7\binom{5}{4}=0
\]
implying that
\[
\dim H^1(X, \mathcal{O}_{X}(D))=0.
\]

We conclude that Theorem~\ref{thm:coh} also holds only in the planar case.

\end{example}

Finally, let us make one remark about rationality---the last remaining of the four conditions in Lemma~\ref{lem:eq}. In the planar case the rationality of $(-1)$ curves follows from the Adjunction formula as explained in Lemma~\ref{lem:eq}. However, in higher dimension a numerical criterium for rationality is difficult to find. For example, it was proved by Castelnuovo that any complex surface with the property that both the irregularity and second plurigenus vanish is rational. 
This criterium is used in the Enriques-Kodaira classification to identify the rational surfaces. In this paper we will not address the rationality question of divisorial $(i)$ classes.

The notion of \emph{$(-1)$ divisors} has been defined previously by Mukai. Our definition is more restrictive when $i=-1$. In order to state Mukai's definition, recall a strong birational map (or pseudo-isomorphism) is an isomorphism outside a set of codimension at least two.

\begin{definition} [Mukai, \cite{Mu}]\label{def:-1mu}
A \textit{$(-1)$ divisor} on $X_{n,s}$ is a divisor $D$ of $X_{n,s}$ for which there exists a strong birational map from $X_{n,s}$ to some $X'$ so that the image of D can be contracted to a smooth point. 
\end{definition}

\begin{remark}\label{rem:-1divweyl}
Since the standard Cremona transformation is a strong birational map, any element of the orbit of the Weyl group action on an exceptional divisor $W_{n,s}\cdot E_i$, is a $(-1)$ divisor.
\end{remark}

Even if the next remark is obvious (see also Remark~\ref{rem:eff}) we will include it here just to emphasize that both definitions of $(-1)$ divisors introduced by Mukai and divisorial $(-1)$ classes are generalizing the notion of $(-1)$ curves in the plane.

\begin{remark}\label{rem:planar}
On the blow up of the projective plane in points the three definitions \emph{divisorial $(-1)$ classes} (Definition~\ref{def:-1div}), \emph{$(-1)$ divisors} (Definition~\ref{def:-1mu})  and \emph{$(-1)$ curves} (Definition~\ref{def:-1 curve}) are equivalent.
\end{remark}

\begin{proposition}\label{prop:sd}
For $i\in \{-1,0,1\}$, let $D$ be a divisorial $(i)$ class on $X_{n,s}$. Then
\begin{enumerate}
\item\label{item:wDstandard} $w(D)$ is also a divisorial $(i)$ class for any Weyl group element $w\in W_{n,s}$.
\item\label{item:wDcones} Cones over $D$ are divisorial $(i)$ classes on $X_{n+1,s+1}$.
\end{enumerate}
\end{proposition}

Part (1) of this Proposition was also appeared in \cite{LUst}, but only when $i=-1$.

\begin{proof}
For \eqref{item:wDstandard} it is enough to prove that $Cr(D)$ is a divisorial $(i)$ class by Remark~\ref{rem:composition}.

Let $D$ be a divisorial $(i)$ class. 
We have seen in Theorem~\ref{thm:cc}
\begin{align*}
\langle \Cr(D), \Cr(D)\rangle&=\langle D, D \rangle=i\\
\adeg(\Cr D)&=\adeg(D)=2+i.
\end{align*}

Property~\eqref{item:CrEffective} of Remark~\ref{rem:cr} implies that $\Cr(D)$ is effective.
The last thing we need to check is that $\Cr D$ is irreducible. Assume by contradiction that $\Cr(D)$ is not irreducible. For $F, G\in \Pic (X_{n,s})$ then
\[
\Cr(D)=F+G.
\]

By Property~\eqref{item:CrInvolution} of Remark~\ref{rem:cr} we have
\[
D=\Cr(F)+\Cr(G).
\]

This a contradiction of the irreducibility of $D$. Thus $\Cr(D)$ is irreducible, and therefore a divisorial $(i)$ class.

We use the same technique to prove \eqref{item:wDcones}. Let $\widetilde{D}$ be the cone over $D$. Indeed by Theorem~\ref{thm:cc},
\begin{align*}
\langle \widetilde{D}, \widetilde{D} \rangle&=\langle D, D \rangle=i\\
\adeg(\widetilde{D})&=\adeg(D)=2+i.
\end{align*}

Obviously, $\widetilde{D}$ is effective because $D$ is effective.
Moreover, assuming that the cone $\widetilde{D}$ is not irreducible, $\widetilde{D}=\widetilde{G}+\widetilde{K}$ then denote by $G$ and $K$ to be the image of $\widetilde{G}$ and $\widetilde{K}$ under projection from the vertex. Then $D=G+K$ contradicting the irreducibility assumption of $D$. Thus the cone $\widetilde{D}$ is a divisorial $(i)$ class.
\end{proof}

In particular, we see that divisorial $(-1)$ classes are also $(-1)$ divisors in the sense of Mukai's definition (see Remark~\ref{rem:-1divweyl}). 
	
We end this section with several examples of important divisors that are not divisorial $(-1)$ classes. It is unknown if they are $(-1)$ divisors in the sense of Mukai, however we expect the first two examples (in Example~\ref{eg:sporadic}) to be generators for the Cox rings of $X_{3,9}$ and $X_{4,14}$, resp. In fact, $D_2$ is related to the Keel--Vermeire divisor on $\overline{M}_{0,6}$, which is known to be a generator of the Cox ring. We call them \emph{sporadic divisors}. The other two examples in Remark~\ref{rem:baselocus} illustrate a point about divisors with base locus.

\begin{example} [Sporadic divisors]\label{eg:sporadic}
Let $D_1\in \Pic(X_{3,9})$ and $D_2\in \Pic(X_{4,14})$ be defined by 
\[
D_1:=2H- E_1-\ldots - E_9, \text{  and } D_2:=2H-E_1-\ldots-E_{14}
\]

We can easily see that  
\begin{align*}
\adeg D_1&=\frac{1}{2}\langle D_1,-K_{X_{3,9}}\rangle = 4 \cdot 2 -9=-1\\
\langle D_1, D_1\rangle &=-1
\end{align*}
and 
\begin{align*}
\adeg D_2&=\frac{1}{3}\langle D_2,-K_{X_{4,14}}\rangle = 5 \cdot 2 -14=-4\\
\langle D_2, D_2\rangle &=-2
\end{align*}
so neither $D_1$ nor $D_2$ is a divisorial $(-1)$ class. 

However, 
\[
\chi (X_{3,9},\mathcal{O}_{X_{3,9}}(D_1))=\binom{5}{2}-9=1.
\]
and
\[
\chi (X_{4,14},\mathcal{O}_{X_{4,14}}(D_2))=\binom{6}{2}-14=1.
\]

Although not divisorial $(-1)$ classes, we see that $D_1$ still satisfies two properties of Lemma~\eqref{lem:eq}, namely
\[
\langle D_1, D_1\rangle =-1\quad \text{  and }\quad \chi (X_{3,9},\mathcal{O}_{X_{3,9}}(D_1))=1
\]
while $D_2$ satisfies just one property of Lemma~\ref{lem:eq}. 

The importance of these divisors lies in the fact that both $D_1$ and $D_2$ are not in the Weyl group orbit of an exceptional divisor $E_i$, but whenever they are contained in the base locus of any divisor $D$, they will create a change in $\dim H^0(X_{n,s},\mathcal{O}_{X_{n,s}}(D))$ in the sense of Conjecture~\ref{conj:GHH} and its corollary.

We would like to investigate in future work if there are other numerical criterium to classify such sporadic divisors on $X_{n,s}$.
\end{example}

\begin{remark}\label{rem:baselocus}
The Alexander--Hirschowitz Theorem classifies all effective divisors $D$ with double points in $X_{n,s}$ for which 
\[\dim H^0(X_{n,s},\mathcal{O}_{X_{n,s}}(D))=\chi (X_{n,s},\mathcal{O}_{X_{n,s}}(d)).
\]
This theorem can also be stated in terms of the secant variety, namely the higher secant variety $\sigma_{s} (V_{n,d})$ to the Veronese embedding $V_{n,d}$ of degree $d$ in $\PP^{n}$  is non-defective with a short list of exceptions. Two of the exceptions are the two divisors $F_1\in \Pic(X_{3,9})$ and
$F_2\in \Pic(X_{4,14})$ given by
\[
F_1:=4H- 2E_1-\ldots - 2E_9, \text{  and } F_2:=4H-2E_1-\ldots-2E_{14}
\]	

Notice that divisors $F_1$ and $F_2$ have the quartics $D_1$ and $D_2$ respectively as base locus. Similar to the base locus Lemma~\ref{lem:base locus lemma} the Dolgachev-Mukai pairing between $D_j$ and $F_j$ is negative

\[
\langle F_1, D_1\rangle =-2 \text{ and } \langle F_2, D_2\rangle =-2.
\]

For both divisors $F_j$ we have $\chi (X_j,\mathcal{O}_{X_j}(F_j))=0$, but they are effective since $F_j=2 \cdot D_j$. Therefore 
\[
\dim H^0(X_j,\mathcal{O}_{X_j}(F_j))=\dim H^1(X_j,\mathcal{O}_{X_j}(F_j))=1.
\]

For arbitrary number of points in higher dimensional projective spaces, no analogue of Conjecture~\ref{conj:GHH} exists for $X_{n,s}$, except in the case $s\leq n+3$; we state this conjecture as Conjecture~\ref{conj:RNC}. 

In general, it is expected that all effective divisors $D$ with $\dim H^1(X_{n,s},\mathcal{O}_{X_{n,s}}(D))\neq 0$ contain base locus. A variety that produces non-vanishing cohomology in degree 1 when contained as the fixed point of a divisor with multiplicity, is called a \emph{special effect variety}. For example, the two divisors $D_1$ and $D_2$ of our previous example are special effect varieties. 
Laface-Ugaglia conjectured in \cite{LU} that the only special effect divisors of $X_{3,s}$ are divisorial $(-1)$ classes and divisors in the Weyl group orbit of $D_1$. It will be interesting to study if sporadic divisors (like $D_j$ or cones over $D_j$)  have  special numerical interpretation similar to divisorial $(-1)$ classes.
\end{remark}

\section{A generalization of the Max  Noether inequality to $\PP^n$.}\label{s:Noether}
In previous sections we emphasized several differences between $(-1)$ curves and divisorial $(i)$ classes.  In this section we will prove Theorem~\ref{thm:noether2}, which we restate here. Recall the classes $\alpha_{k}$ from Section~\ref{ss:weylgroup}.

\begin{remark}
	The Theorem \ref{thm:noether2} generalizes the original Noether's inequality from the projective plane to $\mathbb{P}^n$, following the technique described by Dolgachev \cite{Dolgachev} and \cite[Lemma 5.7.10]{Dolgachev book}, see also \cite[Lemma 2.2]{DO}.
\end{remark}

\begin{theorem*}[=Theorem~\ref{thm:noether2}]
Let $D$ be a divisor with $d\geq m_i\geq 0$ so that $\adeg D=c+2$ and $\langle D, D \rangle=c+e$ for two integers $c$ and $e$ satisfying $-2\leq c, e \leq 1$. If $d=1$ further assume that $\langle D, D \rangle<0$. Then we can reorder the indices, so that  
$\langle	D,\alpha_{k} \rangle\geq 0$ for any $1\leq k\leq s-1$ and  $\langle	D,\alpha_{s} \rangle<0$.

\end{theorem*}


The condition $\langle	D,\alpha_{k} \rangle\geq 0$  says that $m_{1}\geq \ldots \geq m_{s}$ (after reordering) and condition $\langle	D,\alpha_{s} \rangle<0$ says that $m_{1}+m_{2}+ \ldots + m_{n+1} > (n-1) d$. By \eqref{eq:crem}, there is a Cremona transformation such that $\Cr(D)$ has strictly smaller degree than $D$, and also strictly smaller multiplicities at the first $n+1$ points. Thus we say that a divisor $D$ satisfying $\langle	D,\alpha_{s} \rangle<0$ as in the Theorem is not \emph{Cremona reduced}. 

In the case $n=2$, the original hypotheses of the Max Noether inequality for curves---irreducibility, rationality and bounded self-intersection---apply to $(i)$ curves. In contrast, this new result 
says that under good numerical hypotheses of degree and self-intersection (notice the absence of the irreducibility assumption in the hypothesis, so that $D$ is not necessarily a divisorial $(i)$-class), one can perform a Cremona transformation in such a way that reduces the degrees and multiplicities. 

The next example shows that condition $d\geq m_j$ is mandatory in the hypothesis of Theorem~\ref{thm:noether2}. 

\begin{example}\label{eg:ex1} Take $D:=5H- 6E_1-2E_2-E_3-\ldots-E_{13} \in \Pic(X_{3, 13})$. Then 
\begin{equation}
\begin{split}
&\langle D, D \rangle =(3-1)\cdot 5^2-6^2-2^2-11=-1\\
&\adeg D=\frac{\langle D, - K_{X_{3,13}} \rangle}{2} =4\cdot 5-(6+2+11)=1
\end{split}
\end{equation}
We can see that $c=-1$ and $e=0$, however the maximal sum of four multiplicities is 10, and $(n-1)d=10$, so the conclusion of Theorem~\ref{thm:noether2} does not hold. 
The point is $D$ is not effective, and moreover $m_1>d$, so the theorem doesn't apply to this divisor.
\end{example}

\begin{remark}\label{rem1} 
The proof of Theorem~\ref{thm:noether2} can extend also to some cases where $e=2$, but we will leave this to the interested reader. 
\end{remark}

\begin{remark}\label{rem:Notherhypothesis} 
Notice that the irreducibility condition in the hypothesis of the Noether's original is replaced in Theorem~\ref{thm:noether1} by the condition $d\geq m_j\geq 0$. This stronger version of the theorem was originally observed for the planar case in Theorem~\ref{thm:do} of \cite{DO}, however by Remark~\ref{rem:eff} the assumption $d\geq m_j$ can be eliminated only for $n=2$.
\end{remark}

\begin{remark}
Theorem~\ref{thm:noether2} generalizes the Noether inequality even in the case when $n=2$. Indeed, the original Noether inequality (Theorem~\ref{thm:noether1}) requires the curve to be irreducible and rational. The arithmetic genus formula in dimension 2 (Equation~\eqref{eq:agenus}) with $p_a(D)=0$ yields $D\cdot D=\adeg(D)-2$. In the notation of Theorem~\ref{thm:noether2}, this means that $e=0$. On the other hand, using the same notation for $n=2$, we see that $p_a(D)=\frac{e}{2}$. If we consider only elliptic curves, the condition $-2\leq e\leq 1$ forces $p_a(D)=0$, as $p_a(D)$ is a non-negative integer. Again the absence of any irreducibility assumption shows that Theorem~\ref{thm:noether2} is less restrictive. 
Further notice that case $e=2$ of Theorem~\ref{thm:noether2} discussed in Remark~\ref{rem1} generalizes the planar Noether inequality to non-rational divisors $D$.

\end{remark}

We will now dedicate the remaining part of this section to the proof of Theorem~\ref{thm:noether2} generalizing  Max  Noether inequality from $\mathbb{P}^2$ to $\mathbb{P}^n$.

\begin{proof}[Proof of Theorem~\ref{thm:noether2}.]
Case 1. $s\leq n$. Condition $\adeg D =2+c$ implies that $m_{1}+\ldots+m_{s}-(n-1)d=-c+2(d-1)\geq 1$ for $d\geq 2$. If $d=1$, the hypothesis $\langle D, D \rangle<0$ implies $s\geq n$ therefore the statement holds.

Case 2. $s\geq n+1$. We order multiplicities in decreasing order $m_1\geq m_2\geq \ldots\geq m_s$. 

We first assume $d=m_1$---i.e. $D$ is a cone---and prove the statement under this assumption by induction on the dimension $n$. 

The base case is $n=2$. Let $t$ denote the last index with $m_t\neq 0$, i.e. $m_1,\dots, m_t\neq 0$, but $m_{t+1}=\dots=m_s=0$. The conditions $\adeg D=c+2$ and $\langle D,D\rangle =c+e$ imply $-m_2^2-\ldots-m_s^2\geq -4$ and $2m_1-m_2-\ldots-m_s\in \{0,\ldots, 3\}$. This forces $m_2\leq 2$, $t\leq 5$  and $d\leq 3$. We conclude that the only possible cones for $n=2$ satisfying hypothesis conditions are
\begin{itemize}
	\item $d=1$ and $2\leq t\leq 3$; 
	\item $d=m_1=m_2=2$ and $t=2$; 
	\item $d=m_1=2$, $m_2=1$ and $2\leq t\leq 5$;
	\item $d=m_1=3$, $m_2=1$ and $2\leq t\leq 5$.
\end{itemize}

In each of the four cases above the conclusion holds. For $n\geq 3$, let $D=\widetilde{F}$ be a cone over a divisor $F\in \Pic (X_{n-1, s-1})$ of degree $d$ and multiplicities and $m_2,\ldots, m_s$.    
Theorem~\ref{thm:cc} implies that $F$ satisfies hypothesis 
\[
\langle \widetilde{F}, \widetilde{F} \rangle=\langle F, F \rangle=c+e \quad \text{ and }\quad \adeg\widetilde{F}= \adeg F= c+2,
\] 
so by the induction hypothesis $m_{2}+m_{3}+ \ldots + m_{n+1} > (n-2) d.$ Therefore $m_1+m_{2}+ \ldots + m_{n+1} > (n-1) d$, and this concludes the proof for $d=m_1$.  

Finally, we lift the requirement that $d=m_1$. If $d=1$ then condition $\langle D, D \rangle<0$ implies that $m_k=1$ for all $k\leq n$ therefore conclusion holds. We can therefore assume $d\geq 2$, $n\geq 2$ and $d>m_k$ for all $k$. 

For $1\leq j\leq n+1$ define
\[
q_j:=\frac{\sum_{k=j}^s m_k^2}{\sum_{k=j}^s m_k}.
\]

Because $m_j\geq m_k$ for $k\geq j$ we have that $m_j\geq q_j$ for any $1\leq j\leq n+1$. Set
\begin{equation}\label{eq:1}
r_j:=m_j-q_j\geq 0.
\end{equation}

and observe the following equalities
\[
q_1=\frac{(n-1)d^2-c-e}{(n+1)d-c-2} \]
\[q_j=q_{j-1}-r_{j-1}\frac{m_{j-1}}{m_j+\ldots+ m_s}
\]
for any $2\leq j<n+1$.

Recall that by hypothesis $m_1+\ldots+m_s=(n+1)d-c-2$ and $d>m_k$ for all $k$. For every $2\leq j\leq n+1$, since $d\geq m_j+1$ we obtain
\begin{align*}
m_j+\ldots+m_s& =(n+1)d-(m_{1}+\ldots +m_{j-1})-2-c\\
&=(n+2-j)d+(d-m_1)+\dots+(d-m_{j-1})-2-c\\
&\geq (n+2-j)d-c+(j-3).
\end{align*}

From this equality (and $d\geq m_j+1$) we obtain 
\begin{equation*}
\begin{split}
q_j &=q_{j-1}-r_{j-1}\frac{m_{j-1}}{m_j+\ldots+ m_s}\\
&\geq q_{j-1}-r_{j-1}\frac{d-1}{(n+2-j)d-c+(j-3)}.
\end{split}
\end{equation*}

Recall $m_j\geq q_j$, so we obtain
\begin{equation}\label{eq:2}
\begin{split}
m_1+m_2+\ldots+m_{n+1} &\geq (q_1+r_1)+\ldots+(q_n+r_n)+q_{n+1}\\
&\geq(q_1+r_1)+\ldots+(q_{n-1}+r_{n-1})+2q_{n}+r_{n}\Big(1-\frac{d-1}{d-c+(n-2)}\Big)\\
&\geq \sum_{i=1}^{n-2}(q_i+r_i)+3q_{n-1}+r_{n-1}\Big(1-2\frac{d-1}{2d-c+(n-3)}\Big)+
r_{n}\Big(1-\frac{d-1}{d-c+(n-2)}\Big)\\
&\geq (n+1)q_1+\sum_{k=1}^{n} r_k\Big(1-(n+1-k)\frac{d-1}{(n+1-k)d-c+(k-2)}\Big).
\end{split}
\end{equation}

We will now prove that 

\begin{equation}\label{eq:3}
1-(n+1-k)\frac{d-1}{(n+1-k)d-c+(k-2)}=\frac{n-1-c}{d(n+1-k)-c+(k-2)}\geq 0.
\end{equation}
\vskip.3cm 

Indeed, notice that $c\leq 1$ implies $n-1-c\geq 0$ (and equality only for $n=2$ and $c=1$). Moreover, $d\geq 2$ and $k\leq n+1$ imply
\vskip .3cm
\begin{equation*}
\begin{split}
d(n+1-k)+(k-2)-c& \geq 2(n+1-k)+(k-3)\\
&\geq 2n-k-1\\
&\geq 1.
\end{split}
\end{equation*}

Inequalities~\eqref{eq:1},\eqref{eq:2}, and \eqref{eq:3} imply that 
\begin{equation*}
\begin{split}
m_1+\ldots+m_{n+1}&\geq (n+1)q_{1}\\
&=(n+1)\frac{d^2(n-1)-c-e}{(n+1)d-c-2}
\end{split}
\end{equation*}
with equality either if $m_i$ are equal for all $i$ (i.e. $r_1=0$) or if  $n=2$ and $c=1$.

We finally claim that 
\[
(n+1)\frac{d^2(n-1)-c-e}{(n+1)d-c-2}>(n-1)d.
\]

This is equivalent to proving that for all $-2\leq c, e \leq 1$, $n\geq 2$ and $d\geq2$ the following inequality holds
\begin{equation}\label{eq:4}
(c+2)(n-1)d> (c+e)(n+1).
\end{equation}
This follows since
\begin{itemize}
\item If $c=-2$ then $0> (e-2)(n+2)$.
\item If $c=-1$ then $(n-1)d\geq 2(n-1)>0\geq (e-1)(n+1)$.
\item If $c=0$ then $2(n-1)d\geq 4(n-1)>n+1\geq e(n+1)$ since $n\geq2>\frac{5}{3}$.
\item If $c=1$ then $3(n-1)d\geq 6(n-1)\geq 2(n+1)\geq (1+e)(n+1)$. \\
\end{itemize}
Notice equality holds only in the last case $c=e=1$ and $n=d=2$. However, the $\PP^2$ hypothesis implies
$\displaystyle d^2-\sum_{k=1}^s m_k= 4-\sum_{k=1}^s 1=2$  and $\displaystyle 6-\sum_{k=1}^s 1= 3$ therefore $s=3$, so the conclusion holds.
\end{proof}

\section{Generalization of Nagata's correspondence.}\label{s:Nagatacorrespondence}

In this section we will prove Theorem~\ref{thm:weyl3}, that generalizes Theorem~\ref{thm:weyl} (due to Nagata) to $\PP^n$ (see also Remark~\ref{rem:planar}) following the approach of Nagata. In \cite{Dolgachev}, Dolgachev has a nice exposition of Nagata's theorem. Let us first recall Theorem~\ref{thm:weyl3}

\begin{theorem*}[=Theorem~\ref{thm:weyl3}] Let $i\in\{-1,0,1\}$ and $D$ be a divisor in $\Pic(X_{n,s})$. Then $D$ is a divisorial $(i)$ class if and only if it is in the orbit of $H_{n-1-i}$ a hyperplane passing through $n-1-i$ points under the action of the Weyl group. 
In particular, the Weyl group acts transitively on the set of divisorial $(i)$ classes.
\end{theorem*}

It is important to remark that Example~\ref{eg:ex1} and Example~\ref{eg:NotEffective} emphasize the importance of the effectivity assumption for the main theorems of the paper, namely Theorems~\ref{thm:noether2} and the Nagata correspondence in Theorem~\ref{thm:weyl3}.

The first part of the proof of Theorem~\ref{thm:weyl3} follows from the following lemma:

\begin{lemma}\label{lem:exceptional} Let $i\in\{-1,0,1\}$.  A proper transform $H_{n-1-i}$ of a hyperplane  passing through $n-1-i$ points is a divisorial $(i)$ class for $i\in \{-1,0,1\}$.
	
Furthermore, $E_j$ is a divisorial $(-1)$ class for $1\leq j \leq s$.

\end{lemma}

\begin{proof}

For $i\in\{-1,0,1\}$ we have
\begin{align*}
 &\langle H_{n-1-i}, H_{n-1-i} \rangle =i \\
\frac{1}{n-1}&\langle H_{n-1-i}, - K_{X_{n,s}} \rangle =i+2
\end{align*}

Moreover, the proper transform of the hyperplane passing through $n-1-i$ points, $H_{n-1-i}$ is effective and irreducible. The same argument shows that $E_i$ is also a divisorial $(-1)$ class.

\end{proof}

Recall that Theorem~\ref{thm:cc} and Proposition~\ref{prop:sd} imply that the Weyl group preserves intersection pairing of Dolgachev-Mukai and divisorial $(i)$ classes.
In other words, if $w\in W_{n,s}$, then we have 
	\begin{itemize}
		\item $\langle w(D), w(F) \rangle=\langle D, F \rangle$. 
		\item If $D$ is a divisorial $(i)$ class then $w(D)$ is a divisorial $(i)$ class.
	\end{itemize}

This proves the following corollary.

\begin{corollary}\label{cor:weyl2} Let $i\in \{-1,0,1\}$. If $D$ is an $(i)$ Weyl divisor, then $D$ is a divisorial $(i)$ class.

\end{corollary}

For the rest of the proof of Theorem~\ref{thm:weyl3}, it suffices to prove the converse of Corollary~\ref{cor:weyl2} in $X_{n,s}$.

\begin{proof}[Proof of Theorem~\ref{thm:weyl3}] Let $\displaystyle D=dH-\sum_{k=1}^sm_kE_k\in \Pic(X_{n,s})$.  \\
If $D$ is an $(i)$ Weyl divisor, then $D$ is a divisorial $(i)$ class by Corollary~\ref{cor:weyl2}. \\

Conversely, assume that $D$ is a divisorial $(i)$ class on $X_{n,s}$. We prove the statement by induction on $\deg(D)$. 

If $d=0$, the irreducibility assumption implies  $D=E_1$ so $i=-1$ and the result follows from Lemma~\ref{lem:exceptional}.

If $d=1$ all multiplicities are at most $1$ and by irreducibility the self intersection condition 
\[
(n-1)-\sum_{j=1}^s m_j=i
\]
implies that the number of non-zero multiplicities is $n-1-i$. Therefore $D$ is the hyperplane passing through $n-1-i$ points. Again the result follows by Lemma~\ref{lem:exceptional}.

Assume now that $\deg(D)\geq 2$. For convenience order multiplicities increasingly, so that $m_1\geq m_2\geq\dots\geq m_s$. If $s\geq n+1$ then Theorem~\ref{thm:noether2} for $c=i$ and $e=0$ implies that 
\[
m_{1}+\ldots+m_{n+1}> (n-1)d.
\]
Apply a standard Cremona transformation based on points $p_1, \ldots, p_{n+1}$. Equation~\eqref{eq:crem} implies that $\deg (\Cr D)< \deg D$. By Proposition~\ref{prop:sd}, we see that $\Cr D$ is also a divisorial $(i)$ class of smaller degree. The induction hypothesis implies $\Cr D\in W_{n,s}\cdot H_{n-i-1}$, and therefore $D\in W_{n,s}\cdot H_{n-i-1}$, where $H_{n-i-1}$ represents a hyperplane passing through $n-1-i$ points, so it is a $(i)$ Weyl divisor.

Assume $\deg(D)\geq 2$ and $s\leq n$. Since $D$ is an effective divisor, we know $d\geq m_i$ that implies $(n-1)d\geq \sum_{k=1}^{n-1} m_{k}\geq 0$.
Moreover,  Theorem~\ref{thm:noether2} implies that $\sum_{i=1}^n m_i>(n-1)d\geq \sum_{k=1}^{n-1} m_i$. We conclude $s=n$. 


Moreover, since $\sum_{i=1}^{n}m_i-(n-1)d>0$ the Base Locus Lemma 2.1 of \cite{bdp1}, Lemma \eqref{lem:base locus lemma}, we obtain that the hyperplane passing through the $n$ points is in the base locus of the divisor $D$. However, the definition of a divisorial $(i)$ class, implies that $D$ is irreducible. We conclude that $D$ is a hyperplane passing through $n$ points, therefore it is a $(-1)$ Weyl divisor and $\deg(D)=1$.

\end{proof}

\begin{remark}
As Example~\ref{eg:ex1} illustrates, Theorem~\ref{thm:weyl3} does not hold without the effectivity hypothesis of divisorial $(i)$ classes.  This is because the condition of Theorem~\ref{thm:noether2}, namely $d\geq m_k$ for all $1\leq k\leq s$, is not invariant under Cremona transformations. So without the effectivity hypothesis of the divisors, Theorem~\ref{thm:noether2} cannot be applied repeatedly. 
\end{remark}

We now generalize Proposition~\ref{prop:dolg1} (see also \cite{Dolgachev} for $n=2$).

\begin{theorem*}[=Theorem~\ref{thm:dol}] There are no irreducible effective divisors $D$ on $X_{n,s}$ with $ \langle D, D \rangle=r\in \{-3, -2\}$ and $\adeg(D)=-2-r$.
\end{theorem*}

\begin{proof} We assume by contradiction that such divisor $D$ exist. Let $\adeg(D)=-2-r=c+2$, so $c=-r-4\in \{-1,-2\}$. Moreover, $\langle D, D \rangle=r=c+e$ implies $e=r-c=2r+4\in \{-2, 0\}$. We conclude that the divisor $D$ satisfies the hypothesis of Theorem~\ref{thm:noether2} and so the divisor $D$ is not Cremona reduced. We will apply Cremona transformations to reduce the degree and we recognize that $\Cr D$ is an irreducible, effective divisor therefore satisfying the same hypothesis of the theorem.


We apply a sequence of Cremona transformations to transform the divisor $D$ to an effective divisor of minimal degree $w(D)$. If the degree of $w(D)$ is zero, then $w(D)=2E_i$ since $D$ is irreducible and $ \langle D, D \rangle\in \{-3, -2\}$. However the divisor $2E_i$ violates the condition $\adeg(D)=-2-r$, for $r=-2$.
 If $w(D)$ is a divisor degree $1$, so that $\displaystyle w(D)=H-\sum_{i=1}^{s} E_{i}=H_{1\ldots s}$, then the self intersection condition implies that the number of points satisfies $s\geq n+1$ and therefore the divisors $w(D)$ and $D$ are not effective. We obtain a contradiction.
\end{proof}

\begin{corollary}\label{coro}
There are no irreducible effective divisors $D$ on $X_{n,s}$ with $ \langle D, D \rangle=0$ and $\adeg(D)=1$.

\end{corollary}
\begin{proof}
Assume by contradiction such divisor exists. Apply Theorem \ref{thm:noether2} for $c=-1$ and $e=1$ to conclude that $D$ is not Cremona reduced. Therefore, one can apply a Weyl group element until $d=1$, and by effectivity assumption all other multiplicities are at most $1$
 so $w(D)= H-\sum_{k=1}^{q} E_k$. Now $ \langle D, D \rangle=(n-1)-q=0$ implies that $q=n-1$ but $\adeg(D)=(n+1)-(n-1)=2\neq 1$ and this gives a contradiction.
\end{proof}

We remark that in the planar case a smooth curve class $C$ on $X_{2,s}$ with $C\cdot C=0$ and $C\cdot (-K)=1$ would have arithmentic genus $p_a(C)=\frac{2+D\cdot (D+K_X)}{2}=1/2$ that is a contradiction.

\section{Irreducibility Criterium of divisorial $(i)$ classes.}\label{s:irreduc}

In this section, we prove Theorem~\ref{thm:do2}, a  generalization of Theorem~\ref{thm:do} of Dumitrescu--Osserman in \cite{DO}, stating that we can replace the irreducibility assumption by a numerical criterium. Examples~\ref{eg:ex1} and \ref{eg:NotEffective}
show that the assumption that divisors be effective is needed in Theorem~\ref{thm:do2}.

For reference, let us recall Theorem~\ref{thm:do2}. 

\begin{theorem*}[=Theorem~\ref{thm:do2}] Let $i\in \{-1,0,1\}$. The divisor $D$ is a divisorial $(i)$ class on $X_{n,s}$ if and only if $D$ is effective satisfying numerical conditions \eqref{eq:num} and for all degrees $0<d'<d$ and all divisorial $(-1)$ classes $D'$ of degree $d'$ we have $\langle D, D' \rangle\geq 0$.
\end{theorem*}

The proof of the theorem requires the following lemma---a base locus lemma for divisorial  $(-1)$ classes. This is proved in \cite[Corollary 4.5]{LU}, so we omit the proof. 

\begin{lemma}[Base locus lemma \cite{LU}]\label{lem:base locus lemma}
Fix an effective divisor $D$ and let $F$ be a divisorial $(-1)$ class satisfying
\[
-k_F=\langle D, F \rangle< 0.
\]
Then
$F$ is a fixed component of $D$ with multiplicity of containment $k_F>0$.
\end{lemma}


\begin{example}
Lemma~\ref{lem:base locus lemma} requires a divisorial $(-1)$ class. Notice that for $n=2$, the multiplicity of containment of a curve in 
the base locus of a divisor $D\in \Pic(X_{3,9})$ is given by intersection pairing only for $-1$ curves. Indeed, for any positive integer $m$ let us consider the divisor
\begin{equation}\label{eq:multiple}
D:=3mH-mE_1-\ldots-mE_9.
\end{equation}
Take $E:=3H-E_1-\ldots-E_9$ the unique cubic curve passing through nine points. $E$ is an elliptic curve (therefore not a $(-1)$ curve) and notice 
\[
k_E=D\cdot E=9m-9m=0.
\]
By Ciliberto-Miranda \cite{cm2} the Gimigliano-Harbourne-Hirschowitz conjecture \eqref{conj:GHH} holds when the number of points $s$ is a perfect square and when all multiplicities are equal. Therefore
\[\dim H^0(X,\mathcal{O}_X(D))=\chi(X, \mathcal{O}_X(D))=\binom{3m+2}{2}-9\binom{m+1}{2}=1.\]
This implies that projectively $|D|$ contains just one element, and since $mE$ is an element of form \eqref{eq:multiple} we conclude
\[
D=mE.
\]

We conclude that even though $k_E=0$, the elliptic curve $E$ is contained in the base locus of $D$ precisely $m$ times.

\end{example}

We now proceed with the proof of Theorem~\ref{thm:do2}. 

\begin{proof}[Proof of Theorem~\ref{thm:do2}] 
First notice that an effective divisor can never have a negative intersection with a $(i)$-divisorial class with $i\in\{0,1\}$. For example with $i=0$,  assume there exist an effective divisor $D$ say that $\langle D, H_{n-1}\rangle<0$. This is equivalent to the inequality that
\[
(n-1)d-\sum_{i}^{n-1} m_i<0.
\] 
Indeed, $H_{n-1}$ is in particular an effective divisor therefore $d\geq m_i$ for every $1\leq i\leq n-1$, leading to a contradiction. This observation can be generalized to any divisorial $(i)$ class for $i\in\{0,1\}$. 

Proceeding with the proof, assume $D$ is a divisorial $(i)$ class that fails the last condition of the statement of the theorem. We may therefore assume $i=-1$. Then there exists a divisorial $(-1)$ class $D'$ of degree $d'$ smaller than $d$, so that $\langle D, D' \rangle< 0$. By Lemma~\ref{lem:base locus lemma}, $D'$ is a fixed component of $D$ so $D$ can be written as the sum of two divisors $D'$ and $D''$,
\[
D=D'+D''.
\]
This gives a contradiction since $D$ is irreducible.

Conversely, assume that $D$ satisfies the conditions \eqref{eq:num}, we will use induction on $d=\deg D$ to show that $D$ is irreducible.

\begin{enumerate}
\item For the base case $i=-1$, $d=0$, conditions \eqref{eq:num} imply that $m_1=-1$ and all other multipilicities are zero, so $D=E_1$ a $(-1)$ Weyl divisor by Lemma~\ref{lem:exceptional} (a similar argument holds for $d=1$). 
 We remark that $d=0$ and $i\in\{0,1\}$ conditions \eqref{eq:num} have only the trivial solution $D=0$, so we assume that the degree is positive.
\item If $i\in\{-1, 0,1\}$ and $d=1$, one obtains $m_k\geq 0$ with $(n-1)-\sum_{k=1}^s m_k^2=i$ and also $(n+1)-\sum_{k=1}^{s} m_k=2+i$.  If you subtract the two equations you obtain $\sum_{k=1}^{s} m_k \cdot (m_k - 1)=0$ i.e. $m_k\in \{0,1\}$ for all $k\in \{1,\ldots,s\}$.
\end{enumerate}

Therefore we have equality so $s=n-1-i$ and $m_k=1$ for all $k\in \{1,s\}$. For $d=1$ we concluded that $D$ is an $(i)$ divisorial class via Lemma \ref{lem:exceptional}.

Let $d\geq 1$. By induction hypothesis we know the theorem holds for all divisors $D'$ of degree $d'<d$. 
We prove the remainder of the statement contrapositively, namely, assuming $D$ fails the irreducibility condition, we will prove that there exist a divisorial $(i)$ class $D'$ of degree $d'<d$ so that $\langle D, D' \rangle< 0$. By induction, we know this statement holds for divisors of degree smaller than $D$ and we want to prove it for $D$.

If $D$ is not irreducible then there exist $D_1$ and $D_2$---two effective divisors---satisfying
\[
D=D_1+D_2.
\]

We assume all multiplicities of $D$ are non-negative, if not the statement follows trivially.
By Theorem~\ref{thm:noether2} for $c=i$ and $e=0$, there exist indices $i_1,\ldots,i_{n+1}$ so that
\[
m_{i_1}+\ldots+m_{i_{n+1}}> (n-1)d
\]

Applying a standard Cremona transformation based on set $I=\{i_1, \ldots, i_{n+1}\}$, denote by $\overline{D}=\Cr_I D$. Then $\overline{D}$ is effective, satisfies conditions \eqref{eq:num} by Theorem~\ref{thm:cc}, and has smaller degree than $D$. But $\overline{D}$ is not irreducible since 
\[
\overline{D}=\Cr_I(D)=\Cr_I (D_1)+\Cr_I(D_2).
\]
By the induction hypothesis on $\overline{D}$, there exists a divisorial $(-1)$ class $F$, of degree smaller than degree of $D$ so that 
\[
\langle \overline{D}, F \rangle< 0.
\]

We perform a standard Cremona transformation on the index set $I$ and denote by $D':=\Cr_I(F)$. Proposition~\ref{prop:sd} implies that $D'$ is a divisorial $(-1)$ class. Theorem~\ref{thm:cc} implies that 
\begin{align*}
\langle D, D' \rangle&=\langle \Cr\Cr(D), \Cr(F) \rangle\\
&=\langle \Cr(D), F \rangle\\
&=\langle \overline{D}, F \rangle\\
&<0
\end{align*}
This proves the claim.
\end{proof}

The next example shows why irreducibility is needed in Definition~\ref{def:-1div} and why last condition of Theorem~\ref{thm:do2} is needed. This is a generalization
 to dimension three of Example~\eqref{eg:ex1}  of \cite{DO}.

\begin{example}\label{eg:NotIrred} 
Take $i\in\{-1,0,1\}$ and $D\in \Pic(X_{3,8-i})$.
\[
D:=4H-3E_1-3E_2-3E_3-E_4-\ldots-E_{8-i}
\]
Notice that 
\begin{align*}
\langle D, D \rangle&=32-27-(5-i)=i\\
\adeg D &= \frac{\langle D, -K_{X_{3,8-i}} \rangle}{2}=16-9-(5-i)=2+i
\end{align*}
Notice that no divisor in the linear system $|D|$ is irreducible. Indeed, let $H_{123}$ denote the hyperplane passing through the first three points. Lemma~\ref{lem:base locus lemma} implies that hyperplane $H_{123}$ is a fixed component of $D$ since
\[
\langle D, H_{123} \rangle=8-(3+3+3)=-1<0.
\]
Let 
\[
D_1:=D-H_{123}=3H- 2E_1-2E_2-2E_3-E_4-\ldots-E_{8-i}
\]
and notice that $D_1$ (and therefore $D$) is effective since
\begin{equation*}
\begin{split}
\dim H^0(X_{3,8-i}, \mathcal{O}_{X_{3,8-i}}(D_1))-\dim H^1(X_{3,8-i}, \mathcal{O}_{X_{3,8-i}}(D_1))&=\chi(X_{3,8-i}, \mathcal{O}_{X_{3,8-i}}(D_1))\\
&=\binom{6}{3}-3\binom{4}{3}-(5-i)\\
&=3+i>0.
\end{split}
\end{equation*}

\end{example}

\section{Mori Dream Spaces.}\label{s:MDS} 

We end with a few results about Mori Dream Spaces. As before, let $X_{n,s}$ denote blow up of projective space at a collection of $s$ general points and let $W_{n,s}$ be the Weyl group of $X_{n,s}$. It is well known that $X_{n,s}$ is a Mori Dream Space (MDS) whenever $s\leq n+3$; the birational geometry of this space is studied in \cite{AC}, \cite{AM}, \cite{bdp3}, and \cite{Moon}. If $s\geq n+4$, then $X_{n,s}$ is generally not a MDS with the following notable exceptions:

\begin{itemize}
	\item all Del Pezzo surfaces $X_{2, s}$ with $s\leq 8$, 
	\item $X_{3, 7}$ and 
	\item $X_{4, 8}$.
\end{itemize} 
In fact, explicit generators are known for the Cox rings in all of these exceptional cases except the last one (see \cite{park}).

In \cite{CDD} the authors study birational properties of MDS and prove that if $n+1\leq s\leq n+3$, the movable cone of $X_{n,s}$ is the intersection between the Effective cone $\Eff_{\mathbb{R}}(X_{n,s})\subseteq N^1(X_{n,s})_{\mathbb{R}}$, and the dual of the Effective cone $\Eff_{\mathbb{R}}(X_{n,s})^{\vee}$ under the Dolgachev-Mukai pairing.  
In other words,
\[
\Mov(X_{n,s})= \Eff _{\mathbb{R}}(X_{n,s})\cap \Eff_{\mathbb{R}}(X_{n,s})^{\vee}.
\]
This description of the movable cone relies on an interesting description of the orbit $W_{n,s}\cdot E_i$; this description is provided in \cite[Theorem 4.6]{CDD}, which basically says that if $n+1\leq s\leq n+3$, then 
\[
W_{n,s}\cdot E_k= \{D\in \Eff(X_{n,s})| \adeg(D)=1 \}.
\]
The next result, proved in \cite[Theorem 2.7 and 2.8]{CT}, strengthens that theorem to include $s\leq n$.

\begin{theorem}\label{thm:we} Let $s\leq n+3$ with $n\geq 2$. Then the orbit of the Weyl group on $W_{n,s}$ on the exceptional divisor $E_k$ can be described as
	\[
	W_{n,s}\cdot E_k = \{D\in \Eff(X_{n,s})| \adeg(D)=1 \}.
	\]
\end{theorem}

Property~\eqref{item:globalsections} of Remark~\ref{rem:cr} and $\dim H^0(X_{n,s}, \mathcal{O}(E_i))=1$ make the following statement obvious.
\begin{remark} For $X_{n, s}$ then $W_{n,s}\cdot E_i \subset \{D\in \Eff(X_{n,s})| \adeg(D)=1 \}$, the equality of Theorem~\ref{thm:we} holds only for $s\leq n+3$. We give two relevant examples of divisors on MDS $X_{n,s}$ so that $s\geq n+4$ and Theorem~\ref{thm:we} doesn't hold.
	\begin{enumerate}
		\item Consider $\displaystyle D:=3H-\sum_{i=1}^8 E_i$ on $X_{2, 8}$. Divisor $D$ is effective (an elliptic curve) and it has $\adeg(D)=1$ but $D\notin W_{2,8} \cdot E_i$, since $\langle D,D\rangle=1$.
		\item Consider $\displaystyle D:= 2H -\sum_{i=1}^7 E_i$ on $X_{3, 7}$. The divisor $D$ is effective (a quadric surface) and it has $\adeg (D)=1$ but $D\notin W_{3,7}\cdot E_i$, since $\langle D,D\rangle=1$.
	\end{enumerate}
\end{remark}

\begin{remark}
	Theorems~\ref{thm:weyl3} and \ref{thm:we} imply that for $s\leq n+3$, if $D\geq 0$, and $\adeg D =1$, then $D$ is irreducible, and $\langle D,D \rangle=-1$. This is only true for MDS as we have seen. 
	
	A natural question to ask is whether any two of the above conditions imply the other two; we saw a similar phenomenon in Lemma~\ref{lem:eq}. The answer is generally no. For example, if we consider the divisor $D:=3H-3E_1-3E_2-3E_3-E_4$ in $X_{4,4}$. For this divisor, we have $\langle D, D \rangle = - 1$ and $D$ is effective. However, it is easy to see that the divisor $D$ is not irreducible---consisting of the hyperplane through all four points and a quadric double at the first three points---and $\adeg D= \frac{\langle D, -K_{X_{4,4}} \rangle}{3} = 5\neq 1$. 
\end{remark}

We end this section with a result in dimension 2. For a small number of points in two dimensions, the following result is known to hold on $X_{2, s}$ when $s< 9$ (see e.g. \cite{dm2}; for $s=9$ there is an infinite list of $(-1)$ curves). For completeness, we include the proof for $s=9$. Notice that this criterium is much simpler than Theorem~\ref{thm:do} (see for example \cite{DO}).

\begin{lemma}\label{lem:s=9} If $s\leq 9$, a divisor $D$ on $X_{2, s}$ is a $(-1)$ curve if and only if 
	\[
	\langle D, D\rangle=\langle D , K_X\rangle= -1.
	\]
\end{lemma}

The following remark will be useful for the proof. 

\begin{remark}\label{effective cone}
It is known that the only effective cone of divisors for $X_{2,9}$ is tangent to $F$. In other words, for any effective divisor $G$ of $X_{2,9}$ with	
\[
F\cdot G=0 \text{ then there exists } m>0 \text{ so that } G=mF.
\]
	
\end{remark}

\begin{proof}[Proof of Lemma~\ref{lem:s=9}] It is enough to prove that any divisor $D\in\Pic(X_{2,9})$ is irreducible, if 
\[
D^2=-1 \text{ and } K_X \cdot D=-1.
\]
Indeed, Proposition \ref{lem:eq} implies that the divisor $D$ is effective since
$\chi(X, \mathcal{O}_X(D)) =  1$ so $h^1(X, \mathcal{O}_X(D))\geq 1$. 
	
We denote by $F=-K_X$ the anticanonical divisor of $X=X_{2,9}$, consisting of a unique cubic through nine points. Then $F^2=0$ and $F$ is a nef divisor on $X$.
	
Denote the irreducible components of $D$ by $D_i$, i.e. 
\[
D=D_1+\ldots +D_r
\]
Then $F$ is nef and integral and 
\[
1=F \cdot D=\sum F \cdot D_i.
\]

Without loss of generality, we we may assume $D_1$ is the irreducible component of $D$ that meets $F$, i.e. 	
\[
F\cdot D_1=1 \text{ and } F\cdot D_i=0, \text{ for } i\geq 2.
\]
	
The effective divisor $D-D_1\in \Pic(X_{2,9})$ satisfies $(D-D_1)\cdot F=0$, so by Remark \ref{effective cone}, there exists an integer $m>0$ so that 
\[
D-D_1=mF 
\]
	
Hence, we have
\[
D^2=(D_1+mF)^2=D_1^2+2mD_1F+m^2F^2.
\]
	
Now $D_1\cdot F=1$ and $F^2=0$, thus 
\[
-1=D^2=D_1^2+2m
\]
and therefore
\[
D_1^2=-1-2m.
\]	

Since by assumption $K\cdot D_1=-1$, the arithmetic genus formula implies that
\[
p_a(D_1)=\frac{D_1^2-1}{2}+1=\frac{D_1^2+1}{2}\geq 0
\]
We obtain that $D_1^2\geq -1.$
	
Since $D_1^2=-1-2m$ we see that $m=0$ and so $D=D_1$. therefore $D$ is irreducible.

We conclude that $D$ is a $(-1)$ curve on $X=X_{2,9}$.
\end{proof}

\subsection{Moduli Problems.}\label{s:motivation} In this section we will briefly describe further motivation for this work, new directions and connections with other topics, especially related to Mori Dream Spaces. Recall that $X_{n,s}$ is a Mori Dream space for $s\leq n+3$, as well as the spaces $X_{3,7}$ and $X_{4,8}$ and $X_{2,s}$ for $s\leq 8$. 

The birational geometry of Mori Dream Spaces can be encoded in finite data, namely rational polyhedral cones together with their Mori chamber decomposition. 
The chamber decomposition is determined by arbitrary dimensional $(-1)$ classes,  
which, unlike $(-1)$ curves or divisorial $(-1)$ classes, live on the blow up of $X_{n,s}$ along different subvarieties. 
Thus studying such classes can be quite difficult. We will now discuss two connections of the Mori Dream Spaces $X_{n,n+3}$ to other interesting constructions, beginning with the moduli space of certain vector bundles. 
This point of view is discussed, e.g. in \cite{Bau}, \cite{Mu05}, and later in \cite{AC}, \cite{AM}, \cite{bdp3}, \cite{Cas1}, \cite{Cas2}, \cite{CT}, \cite{Moon}.

The geometry of the Mori Dream Space $X_{n,n+3}$ was studied first by Mukai and Bauer due to it's connection 
to the moduli space of rank 2 parabolic vector bundles over $\PP^1$. The birational geometry of the moduli space of rank 2 semistable parabolic vector bundles on a rational curve, in particular the effective cone and all birational models (that correspond to moduli space of parabolic vector bundles with certain weights) was studied in \cite{Moon}.

In order to see the connection, we first fix $n+3$ points $p_1, \ldots, p_{n+3}$ in $\PP^1$. A parabolic rank 2 vector bundle over $\PP^1$ is comprised of the following data:
\begin{itemize}
	\item A vector bundle $E$ of rank 2 over $\PP^1$;
	\item for each $k\in \{1, 2, \dots, n\}$, a 1-dimensional subspace $V_k\subset E|_{p_k}$;
	\item a sequence of rational numbers $\vec{a}=(a_1,\dots,a_n)$ such that $0\leq a_k<1$. 
\end{itemize}

The parabolic slope of $E$ is defined to be
\[
\pardeg E:=\frac{\sum_{k=1}^n a_k}{2}
\]
Two rank 2 semisimple bundles are $S$-equivalent if they have the same factors in their Jordan-Holder filtration.
Denote by $\mathcal{M}_{\vec{a}}$  the moduli space of rank 2, degree 0 equivalence classes (under $S$-equivalence) of semistable parabolic vector bundles 
over $\PP^1$, with parabolic structure $\vec{a}$.

Bauer in \cite{Bau} introduced weight polytope $\Delta\subset [0,1]^{n+3}$---i.e. weights for which the moduli space is non-empty---and studied the chamber decomposition of this space. 
The moduli space $\mathcal{M}_{\vec a}$ varies with weights; for some weights it could be either empty, or it could be $\PP^n$ or $X_{n,n+3}$. 
In fact, Bauer in \cite{Bau} and Mukai in \cite{Mu05} proved that $X_{n,n+3}$ is isomorphic to $\mathcal{M}_{\vec{a}}$ where $\vec{a}=(\frac{1}{n},\ldots, \frac{1}{n})$. 

The GIT information is encoded in this polyhedral data; the chamber containing the central weight $(\frac{1}{2}, \frac{1}{2}, \ldots, \frac{1}{2})\in \Sigma$   
determines the other chambers. 
Furthermore, Mukai in \cite{Mu05} and Araujo and Massarenti in \cite[Theorem 3.4]{AM} proved that there exist a linear projection $\pi:\mathbb{R}^{n+4}\stackrel{}{\rightarrow} \mathbb{R}^{n+3}$ so that \begin{itemize}
	\item $\pi(\Eff(X_{n,n+3}))=\Delta$, 
	\item $\pi(\Mov(X_{n,n+3}))=\Sigma\subset \Delta$, 
	\item $\pi(-K_{X_{n,n+3}})=(\frac{1}{2}, \ldots, \frac{1}{2})$ 
\end{itemize}
and so that the Mori chamber decomposition of $\Eff(X_{n,n+3})$ induces the chamber decomposition of $\Delta$.

Another interesting birational model for $X_{n,n+3}$ when $n$ is even is constructed as follows. Define two quadrics $Q_1$ and $Q_2$ by the following two equations:
\[\sum_{k=1}^{n+3}x_k^2=0 \qquad\sum_{k=1}^{n+3} \lambda_kx_k^2=0
\]
where $\lambda_k$ are some fixed complex numbers. 

Let $Z:=Q_1\cap Q_2\subset \PP^{n+2}$ denote their complete intersection. Set $n=2m$ and let $G$ be the subvariety of the Grassmannian $Gr(m-1,\PP^{n+2})$ parametrizing linear cycles of dimension $m-1$ contained in $Z$.  
Then $G$ is a smooth $n$ dimensional Fano variety with Picard number $n+4$ and is isomorphic to $\mathcal{M}_{\vec{a}}$ where $\vec{a}=(\frac{1}{2}, \frac{1}{2}, \ldots, \frac{1}{2})$ (see \cite{Cas1}). Bauer in \cite{Bau} and Casagrande in \cite{Cas1} proved that varieties $G$ and $X_{n,n+3}$ are strongly birational (i.e. isomorphic in codimension $1$). 

Furthermore, let us denote by $\mathcal{M}:=\{M\cong\PP^m| M\subset Z\}$ the finite set (with cardinality $2^{2m+1}$) of cycles of dimension $m$. 
For every $M\in \mathcal{M}$, define $E_{M}:=\{L=\PP^{m-1}|L\cap M\neq \emptyset\}\in \Eff(G)$. The variety $Z$ has an involution $\sigma_k$ sending $x_k$ to $-x_k$ for each $k\in\{1,\dots,n+3\}$. Now fix an element $M\in\mathcal{M}$, and set $M_k=\sigma_k(M)$.  
Araujo-Casagrande 
in \cite{AC} prove that for a fixed choice of $M$, there is a unique birational map $\rho_M:G\to \PP^n$ inducing a strong birational map $G\to X_{n, n+3}$ with the properties that under this map $E_{M_{k}}$ are sent in the divisorial $(-1)$ classes while  
$E_M$ is sent to the secant variety to the unique rational normal curve of degree $n$ passing through all $n+3$ points in $\PP^n$.

The combinatorial data describing $X_{n,n+3}$ as a Mori Dream space was also independently identified in \cite{bdp3} by Brambilla-Dumitrescu-Postinghel from a different point of view. Even if extremal rays of the Effective and Movable cones of divisors are known, finding the facets of cone with given rays is in general a difficult combinatorial problem. For $X_{n,n+3}$, in \cite{bdp3} the authors emphasize that computation of the facets of the Effective and Movable cones of divisors can be computed via geometry determined by divisorial $(-1)$ classes. More precisely, Theorem~5.1 of \cite{bdp3} highlights that the facets of the Effective cone of divisors of $X_{n,n+3}$ (or Movable cones of divisors) are obtained from the Dolgachev-Mukai pairing with divisorial $(-1)$ classes. 

Currently there is no general definition as in Equation \eqref{eq:num} for arbitrary dimensional $(-1)$ classes; numerical conditions are not known and therefore no rigorous examples on how to construct them. However, in \cite{bdp3} the authors emphasize that cones over the secant variety of a rational normal curve $J(L_I, \sigma_t)$ dimension $|I|+2t-1$ (see Section \ref{s:MDS}) are such $(-1)$ classes determining the chamber decomposition of the movable cone.  
Also the Mori chamber decomposition of these cones into nef chambers is given by similar conditions where walls are determined by such $(-1)$ classes of arbitrary dimension.
So far there is no explanation for why this should be true.

This leads us to the following intuition of how one might define $(-1)$ classes of arbitrary dimension, which gives further motivation to understand the divisorial $(-1)$ classes. We consider $(-1)$ classes to be irreducible components the intersections of distinct divisorial $(-1)$ classes that are orthogonal with respect to the Dolgachev-Mukai pairing. 

In \cite{LU}, Laface and Ugaglia define \emph{elementary $(-1)$ curves} to be the orbit of a line through two points under the Weyl group action. In dimension 3, this is equivalent to the intuitive definition we have just described, however our intuitive definition can be extended to higher dimension.

Let us take $X_{3,6}$ as an example. Because we are in dimension 3, we consider divisorial $(-1)$ classes, and $(-1)$ classes of dimension 1, which we will call $(-1)$ curves. A $(-1)$ curve 
be an irreducible curve contained in the intersection of two divisorial $(-1)$ classes on $X_{3,6}$. As we will later show, this means a $(-1)$ curve is of the form $w(E_i) \cdot w(E_j)$ for $w\in W_{3,6}$ the Weyl group of $X_{3,6}$ as in Remark 2.6. In this example, the only divisorial $(-1)$ classes on $X_{3,6}$ are the exceptional divisors $E_l$, hyperplanes through three points $H_{pjk}$ (i.e. the divisor $H-E_p-E_j-E_k$) and cones over the conic through five points in $\PP^2$, (i.e. $2H - 2E_p-\displaystyle\sum_{j\neq p} E_j$). 

In other words, a $(-1)$ curve is either a line through two points  
\[
L_{pj}=H_{pjk} \cdot H_{pjl}
\]
or is an intersection between two cones as described, namely for some $i, j\in \{1,\ldots, 6\}$ and $i\neq j$,
\begin{equation}\label{eq:cones}
\begin{split}
&D_k:=2H - 2E_k-\sum_{p\neq k} E_p\\
&D_j:=2H -2E_j-\sum_{p\neq j} E_p.
\end{split}
\end{equation}

By the theorem of B\'ezout, both divisors $D_i$ and $D_j$ contain in their base locus the line $L_{kj}$ passing through the first 2 points. Therefore the intersection between $D_k$ and $D_j$ is a reducible curve of degree $2\cdot 2=4$ and passing through points $p_k$ and $p_j$ with multiplicity $2\cdot 1=2$, and passing through the last four points (multiplicity $1 \cdot 1=1$). Therefore
\[
D_k \cdot D_j = L_{kj}+C.
\]
and so we deduce that $C$ has to be the unique rational normal curve in $\PP^3$ of degree $4-1=3$ passing through all six points.

Note that the other intersections $D_{p} \cdot H_{pjk}$ of orthogonal divisors with respect to Dolgachev-Mukai pairing are reducible and can be expressed as sum of two lines.

From this discussion we see that the birational geometry of $X_{3,6}$ encoded in polyhedral cones is fully determined by $26=\binom{6}{3}+ \binom{6}{1}$ divisorial $(-1)$ classes together with sixteen 
$(-1)$ curves ($16=\binom{6}{2}+1$), which are the intersection of divisorial $(-1)$ classes. 

We can use this point of view to understand the Mori Dream Space $X_{3,7}$ as well, which falls outside of the discussion that began this section. 
We obtain $X_{3,7}$ by blowing up one more point in $X_{3,6}$. Here, as shown in \cite{AC}, the birational geometry is determined by $(-1)$ divisorial classes ($119$ of them) together with lines through 2 points and rational normal curves of degree $3$ passing through $6$ points ($28$ of these).

Since the birational geometry of space $X_{n, n+3}$ is understood, the only remaining mysterious case of a Mori Dream Space is $X_{4,8}$. Its birational geometry is fully determined by divisorial $(-1)$ classes and by $(-1)$ curves, but also by special surfaces that we will call \emph{$(-1)$ surfaces}. A finite list of $(-1)$ surfaces together with some of their properties was identified in \cite{Cas2}; we will roughly present it here.

We will begin by relating the geometry of $X_{2,8}$ with $X_{4,8}$, as in \cite{Mu05}. 
More specifically, let $X_{2,8}$ be the blow up of $\PP^2$ at points $q_1,...,q_8$ and $X_{4,8}$ be the blow up of $\PP^4$ at $p_1,...,p_8$ general points. These two varieties are connected by Gale duality
giving a bijection between sets of 8 general points in $\PP^2$ and in $\PP^4$, up to projective equivalence. The precise relation between $X_{2,8}$ and $X_{4,8}$ was established in the following theorem of Mukai:
\begin{theorem}[\cite{Mu05}, see also \cite{Cas2}]
	$X_{4,8}$ is isomorphic to moduli space of rank $2$ torsion free sheaves $F$ on $X_{2,8}$ for which $c_1(F)=-K_S$ and $c_2(F)=2$. 
\end{theorem}

In this theorem, the semistability refers to semistability in the sense of Gieseker-Maruyama with respect to $-K_{X_{2,8}}+2h$ where $h\in Pic(X_{2,8})$ is the pull-back of $\mathcal{O}_{\PP^2}(1)$ under the map
$X_{2,8}\rightarrow \PP^2$.

Mukai's proof of this theorem is based on the study of the birational geometry of the moduli space of such rank 2 torsion free sheaves in terms of the variation of the stability condition given by the ample line bundle of $-K_{X_{2,8}}+2h$. 

By studying the geometry of the del Pezzo surface $X_{2,8}$ and transcribing via Mukai's correspondence, the authors in \cite{Cas2} describe the five types surfaces in $X_{4,8}$ playing a special role in the Mori program: 
\begin{enumerate}
	\item planes passing through three points;
	\item cones over one point of the rational normal curve in $\PP^3$ passing through seven points;
	\item surfaces of degree 6 with three simple points and five triple points;
	\item surfaces of degree 10;
	\item surfaces of degree 15.
\end{enumerate}
Each of these surfaces comes equipped with certain multiplicities at the eight points.

However, following the point emphasized in this work, we can explicitly construct {\it Weyl cycles} on $X_{4,8}$ as irreducible surfaces contained in the intersection of two divisorial $(-1)$ classes on $X_{4,8}$ that are orthogonal with respect to the Mukai pairing as in \cite{bdp4}. 
This definition will agree with and extend the $(-1)$ Weyl lines of Laface Ugagila \cite{LU}. For example we can take the cones over divisors \eqref{eq:cones} in $X_{3,7}$ with the vertex consisting at one point $p_k$ 
\[
D_{k,q}:=2H - 2E_k - 2E_q-\sum_{p\neq i} E_p
\]
\[
D_{k,j}:=2H - 2E_k - 2E_j-\sum_{p\neq j} E_p
\]

As in the previous example \eqref{eq:cones}, we can see the plane determined by points $L_{kqj}$ is in the base locus of both divisors $D_{k,q}$ and $D_{k,j}$. Indeed, one can see by  B\'ezout's theorem that both divisors contain a pencil of lines $L_{ks}$ passing through the cone of each divisor and through an arbitrary point point $s$ on line $L_{qj}$.

$D_{k,i}$ and $D_{k,j}$ are two divisorial $(-1)$ classes orthogonal with respect to the Mukai pairing, that intersect along a union of two surfaces of degree $2\cdot 2$ and multiplicities $m_k=2\cdot2=4$, $m_i=2$ and $m_j=2$ while all other multiplicities are $m_p=1\cdot 1=1$. Since plane $L_{kij}$ is in this intersection, the residual surface is of degree $3$---i.e.
a cone at vertex $P$ over the rational normal curve in $\PP^3$ of degree 3 passing through 6 points, $J(P, \sigma_1(C))$. 

In a similar way one can express explicitly all possible {\it Weyl surfaces} using intersection theory in the Chow Ring $A^2(X_{4,8})$ of $X_{4,8}$ along all lines and all rational normal curves of degree 4 passing through 7 points. A rigorous definition and a classification for $(-1)$ curves and $(-1)$ surfaces on $X_{4,8} $ was given in \cite{bdp4} and \cite{dm1}.
Surprisingly, these {\it Weyl surfaces} correspond to the ones of \cite{Cas2}. Moreover, these {\it Weyl surfaces} are also the Weyl group orbit of a plane through 3 fixed points (see \cite{dm1}). However, it is not known if there is an arithmetic definition similar to \eqref{eq:num} for arbitrary dimensional $(-1)$ classes on $X_{n,s}$ (and their relations with $F$ strata on $\overline{\mathcal{M}}_{0,n}$).

The current work also is related to understanding the Cox ring of $\overline{\mathcal{M}}_{0,n}$. In \cite{kap} Kapranov identified $\overline{\mathcal{M}}_{0,n}$ with a projective variety isomorphic to the projective space $\PP^{n-3}$ successively blown up along $r$-dimensional cycles spanned by $(r+1)$-subsets of a set with $n-1$ general points, with $r$ increasing from 0 to $n-4$. Though we only look at blowing up in a collection of points, there seems to be a relationship between what we will call "sporadic divisors" and generators of the Cox ring of $\overline{\mathcal{M}}_{0,n}$. For example, the first sporadic divisor $D_1=2H-E_1-E_2-\dots - E_9$ discussed in Example~\ref{eg:sporadic} is very similar to the Keel-Vermiere divisor  that is a quadric through 5 points and 4 lines (see e.g. \cite{PV}). The Keel-Vermiere divisor together with the boundary divisors are generators of the Cox ring of
$\overline{\mathcal{M}}_{0,6}$ (see \cite{C}). 
A similar statement is true for the second sporadic divisor in Example~\ref{eg:sporadic}. 
Furthermore, $\overline{\mathcal{M}}_{0,n}$ has been proven to NOT be a Mori Dream Space for $n\geq 10$ (see e.g. \cite{CT2}, \cite{GK}, \cite{HKL}). 

This brings up two questions of interest. First, can we expect a correspondence in the world of Cox ring generators of $\PP^{n-3}$ blown up in $s$ points and of $\overline{\mathcal{M}}_{0,n}$ by some sort of degeneration argument?

And secondly, one can ask if the notion of divisorial $(-1)$ classes can be generalized to the blow up of $\PP^n$ in higher dimensional cycles, as described above, and whether the generators of the Cox ring can be described by numerical conditions as discussed in the current article.  This would give some insight into the generators of Cox rings of $\overline{\mathcal{M}}_{0,n}$ through a combinatorics point of view. This problem will be the topic of future work. 

\subsection{A conjecture on $X_{n, n+4}$}
\label{s:conjecture} In general dimension $n$, few things about classical interpolation problems in $\PP^n$ are known. For example in the three dimensional space $X_{3,s}$ there is a Question similar to Conjecture~\ref{conj:GHH} formulated by Laface and Ugaglia in \cite{LU}. Surprisingly, the mysterious quadric $D_1$ analyzed in Example~\ref{eg:sporadic} plays a crucial role there. 

Even if it is difficult to formulate Conjecture~\ref{conj:GHH} in arbitrary dimension $n$ and general $s$, in \cite{bdps2} the following question is stated for $s=n+4$ points. We will briefly describe its flavor below.

Let us denote by $W(r)$ the collection of all $r$-Weyl planes in $X_{n,s}$, namely the elements in the orbit of the cycle of dimension $r$ passing through a collection of $r+1$ fixed points of $X_{n,s}$, under the Weyl group action. Recall, if $X_{n,s}$ is not a Mori Dream Space, its Weyl group is infinite (see \cite{dm2, dm3}), and moreover its orbit of a fixed $r$-cycles is also infinite. For a Weyl $r$-plane $W$ and an effective divisor $D$ in $\Pic(X_{n,s})$, let us denote by $k_{W}(D)$ the multiplicity of containment of the cycle $W$ in the base locus of the divisor $D$.

Let $W=w(L_I)$ be an $r$-Weyl plane obtained by the action of an element $w$ of the Weyl group $W_{n,s}$ on a linear cycle $L_I$ spanned by a set $I=\{p_{1}, \ldots, p_{r+1}\}$ of cardinality $|I|=r+1$.
Consider the general line class $h\in X_{n,s}$ and let $e_{i}$ denote a general line class in the exceptional divisor $E_i$. Recall the Chow group $A^{n-1}(X_{n,s})$ is generated by the classes $h$ and $e_i$ for $i=\{1, \ldots, s\}$. Further denote $c=rh-e_1-\ldots - e_{r+1}\in A^{n-1}(X_{n,s})$. Then the multiplicity of containment of the cycle $W$ in $D$ can be computed (see e.g. \cite{bdps2}) by intersection product 
\[
k_W(D)=\max\{0, -D\cdot w(c_I).\}
\]
Here the curve class $w(c_I)$ 
is computed via Weyl group action on curve classes $A^{n-1}(X_{n,s})$ as in \cite{dm1, dm2}.

\begin{conjecture}[\cite{bdps2}]\label{conj:RNC}
If $s\leq n+4$ and $D$ is an effective divisor in $\Pic(X_{n,s})$ then
\[
\dim H^0(X_{n,s}, \mathcal{O}(D))=\chi(X_{n,s}, \mathcal{O}(D))+ \sum_{r=1}^{n-1}\sum_{W\in W(r)}(-1)^{r+1}{{n+k_{W}(D)-r-1}\choose n},
\]
where the sum ranges over all the r-Weyl planes $W$ contained in the base locus of a divisor $D$. 
\end{conjecture}

Conjecture \eqref{conj:RNC} was proved to hold for $s=n+2$ in \cite{bdp1}, and for $s=n+3$ in \cite{bdps2} and \cite{LPS}. However, the case when $s=n+4$ is an open problem. On the other hand, Conjecture~\eqref{conj:RNC} cannot hold in $X_{n,s}$ for general $s$ and general $n$ (see Example~\ref{eg:sporadic}) due to the quadric passing through $9$ fixed points in $\mathbb{P}^3$, as observed by \cite{LU}.

\end{document}